\documentclass[a4paper, 12pt]{article}

\usepackage[a4paper,top=3cm,bottom=3cm,left=3cm,right=3cm]{geometry}
\usepackage{lmodern}

\usepackage[english]{babel}
\usepackage[T1]{fontenc}
\usepackage[utf8]{inputenc}

\usepackage{authblk}

\usepackage[normalem]{ulem}
\usepackage{tipa}

\usepackage{csquotes}
\usepackage[hidelinks]{hyperref}	
\usepackage{url}
\hypersetup{    
	colorlinks = true,
	linkcolor  = black,
	urlcolor   = gray,
	citecolor  = black
}
\usepackage[]{natbib}
\let\cite\citep
\usepackage{breakcites}

\usepackage{titletoc}

\usepackage[acronym]{glossaries}
\makeglossaries

\usepackage{enumitem}
\usepackage{wasysym}

\usepackage{graphicx}
\usepackage{subcaption}
\usepackage{rotating}
\usepackage{wrapfig}
\counterwithin{figure}{section}

\usepackage[table,xcdraw,dvipsnames]{xcolor}
\usepackage{booktabs}
\usepackage{multirow}
\usepackage{adjustbox}
\counterwithin{table}{section}

\definecolor{highlight}{HTML}{4F7CB0}
\definecolor{pass-through}{HTML}{00346F}
\definecolor{inverse-pass-through}{HTML}{4F7CB0}
\definecolor{self-loop}{HTML}{C95100}
\definecolor{inverse-self-loop}{HTML}{FC9432}

\usepackage{amssymb}
\usepackage{esint}
\usepackage{amsmath}

\usepackage{amsthm}

\theoremstyle{definition}
\newtheorem{definition}{Definition}[section]

\theoremstyle{plain}
\newtheorem{theorem}[definition]{Theorem}
\newtheorem*{theoremNoNb}{Theorem}
\newtheorem{lemma}[definition]{Lemma}
\newtheorem*{lemmaNoNb}{Lemma}
\newtheorem{corollary}[definition]{Corollary}
\newtheorem*{corollaryNoNb}{Corollary}

\theoremstyle{remark}
\newtheorem{remark}[definition]{Remark}

\theoremstyle{plain}

\allowdisplaybreaks
\usepackage{bbold}

\usepackage{listings}
\usepackage{algorithm} 
\usepackage{algpseudocode}

\usepackage[super]{nth}
\usepackage{lipsum}

\newglossaryentry{formula}{name=formula,
                           description={A mathematical expression}}

\newacronym{wrt}{w.r.t.}{with respect to}
\newacronym{EU}{EU}{European Union}
\newacronym{iid}{iid}{independent and identically distributed}
\newacronym{iff}{iff}{if and only if}
\newacronym{wlog}{w.l.o.g.}{without loss of generality}
\newacronym{Wlog}{W.l.o.g.}{Without loss of generality}
\newacronym{SLE}{SLE}{system of linear equations}
\newacronym{SNLE}{SNLE}{system of non-linear equations}
% ...
\newacronym{CDF}{CDF}{cumulative distribution function}
% ...
\newacronym{OP}{OP}{optimization problem}
\newacronym{AI}{AI}{artificial intelligence}
\newacronym{XAI}{XAI}{explainable artificial intelligence}
\newacronym{ML}{ML}{machine learning}
\newacronym{DL}{DL}{deep learning}
\newacronym{GDL}{GDL}{geometric deep learning}
\newacronym{SOTA}{SOTA}{state of the art}
\newacronym{TP}{TP}{true positives}
\newacronym{FP}{FP}{false positives}
\newacronym{FN}{FN}{false negatives}
\newacronym{TN}{TN}{true negatives}
\newacronym{ACC}{ACC}{accuracy}                     
\newacronym{TPR}{TPR}{true positive rate}           
\newacronym{FPR}{FPR}{false positive rate}          
\newacronym{FNR}{FNR}{false negative rate}          
\newacronym{TNR}{TNR}{true negative rate}           
\newacronym{PPV}{PPV}{positive predictive value}    
\newacronym{FDR}{FDR}{false discorvery rate}        
\newacronym{FOR}{FOR}{false omission rate}          
\newacronym{NPV}{NPV}{negative predictive value}    
\newacronym{ROC}{ROC-curve}{receiver operating characteristic curve}
\newacronym{AUC}{AUC}{area under the (ROC) curve}
\newacronym{IG}{IG}{information gain}
\newacronym{MSE}{MSE}{mean squared error}
\newacronym{AE}{AE}{absolute error}
\newacronym{MAE}{MAE}{mean absolute error}
\newacronym{MRAE}{MRAE}{mean relative absolute error}
\newacronym{AC}{AC}{absolute cost}
\newacronym{MAC}{MAC}{mean absolute cost}
\newacronym{DI}{DI}{disparate impact}
\newacronym{DP}{DP}{demographic parity}
\newacronym{EOs}{EOs}{equalized odds}
\newacronym{EO}{EO}{equal opportunity}
\newacronym{SVM}{SVM}{support vector machine}
\newacronym{MLP}{MLP}{multi layer perceptron}
\newacronym{LLM}{LLM}{large language models}
\newacronym{NN}{NN}{neural network}
\newacronym{PINN}{PINN}{physics-informed neural network}
\newacronym{GNN}{GNN}{graph neural network}
\newacronym{MPNN}{MPNN}{message passing neural network}
\newacronym{GCN}{GCN}{graph convolutional network}
\newacronym{PI-GNN}{PI-GNN}{physics-informed graph neural network}
\newacronym{PI-GCN}{PI-GCN}{physics-informed graph convolutional network}
\newacronym{ADR-GNN}{ADR-GNN}{advection-diffusion-reaction graph neural network}
\newacronym{QAMP}{QAMP}{quantum advection message passing}
\newacronym{MeGA-MP}{MeGA-MP}{metric graph advection message passing}
% ...
\newacronym{WL}{WL}{Weisfeiler-Lehmann}
\newacronym{WDN}{WDN}{water distribution network}
\newacronym{WDS}{WDS}{water distribution system}
\newacronym{ERC}{ERC}{European Research Council}
\newacronym{PI-Algo}{PI-algorithm for hydraulic state reconstruction}{physics-informed algorithm for hydraulic state reconstruction}
\newacronym{ODE}{ODE}{ordinary differential equation}
\newacronym{PDE}{PDE}{partial differential equation}
\newacronym{ADR}{ADR}{advection-diffusion-reaction}
\newacronym{ADR-PDE}{ADR-PDE}{advection-diffusion-reaction partial differential equation}
\newacronym{MoC}{MoC}{Method of Characteristics}

\newcommand{\n}[1]{{n_{\text{#1}}}}                     
\newcommand{\N}{\mathbb{N}}                             
\newcommand{\R}{\mathbb{R}}                             
\newcommand{\1}{\mathbb{1}}                             
\newcommand{\Idm}[1]{\mathbf{1}_{#1 \times #1}}         
\newcommand{\Zerom}[2]{\mathbf{0}_{#1 \times #2}}       
\newcommand{\Dm}{\mathbf{D}}                            
\newcommand{\zerom}[1]{{\mathbf{0}_{#1}}}               
\newcommand{\B}{\mathcal{B}}                            
% ...
\newcommand{\xm}{\mathbf{x}}                            
\newcommand{\hm}{\mathbf{h}}                            
% ...
\newcommand{\bm}{\mathbf{b}}                            
% ...
\newcommand{\G}{\mathcal{G}}                            
\newcommand{\Nbh}{\mathcal{N}}                          
\newcommand{\Nbhin}{{\Nbh_{\text{in}}}}                 
\newcommand{\Nbhout}{{\Nbh_{\text{out}}}}               
\newcommand{\Bm}{\mathbf{B}}                            
\newcommand{\Bmtilde}{\tilde{\Bm}}                      
\newcommand{\dia}{\rho}                                 
\newcommand{\dm}{\mathbf{d}}                            
\newcommand{\qm}{\mathbf{q}}                            
\newcommand{\Vr}{{V_{\text{r}}}}                        
\newcommand{\Vc}{{V_{\text{c}}}}                        
\newcommand{\Ei}{{E_{\text{i}}}}                        
\newcommand{\Ed}{{E_{\text{d}}}}                        
\newcommand{\Vh}{{V_{\text{h}}}}                        
\newcommand{\Vd}{{V_{\text{d}}}}                        
\newcommand{\Eq}{{E_{\text{q}}}}                        
% ...

% ...
\DeclareMathOperator*{\argmax}{argmax}                  
\DeclareMathOperator{\sgn}{sgn}                         
\DeclareMathOperator{\im}{im}                           
\DeclareMathOperator{\rk}{rk}                           
\DeclareMathOperator{\id}{id}                           
% ...

% ...

% ...

\title{
    \Large 
    \textbf{Existence and Uniqueness of Physically Correct Hydraulic States in Water Distribution Systems}
    \\
    \small
    \textbf{A theoretical analysis on the solvability of non-linear systems of equations in the context of water distribution systems}
}
\author{
    \small
    Janine Strotherm, 
    Julian Rolfes 
    and
    Barbara Hammer
}
\affil{
    \small
    Bielefeld University
    \\
    \texttt{\{jstrotherm,bhammer\}@techfak.uni-bielefeld.de}
    \\
    \texttt{julian.rolfes@uni-bielefeld.de}
}
\date{
    \tiny 
}

\begin{document}

\maketitle


\begin{abstract}
Planning and extension of water distribution systems (WDSs) plays a key role in the development of smart cities, driven by challenges such as urbanization and climate change. 
In this context, the correct estimation of physically correct hydraulic states, i.e., pressure heads, water demands and water flows, is of high interest. 
Hydraulic simulators such as EPANET or more recently, physic-informed surrogate models are used to solve this task.
They require a subset of observed states, such as heads at reservoirs and water demands, as inputs to estimate the whole hydraulic state.
In order to obtain reliable results of such simulators, but also to be able to give theoretical guarantees of their estimations, an important question is whether theoretically, the subset of observed states that the simulator requires as an input suffices to derive the whole state, purely based on the physical properties, also called hydraulic principles, it obeys. 
This questions translates to solving linear and non-linear systems of equations.
Previous articles mainly investigate on the existence question under the term \textit{observability analysis}, however, they rely on the approximation of the non-linear principles using Taylor approximation and on network-dependent numerical or algebraic algorithms that need to be solved in order to give answer to the question whether a subset of observed states determines the whole state. 
In this work, we provide purely theoretical guarantees on the existence and uniqueness of solutions to the non-linear hydraulic principles, and by this, the existence and uniqueness of physically correct states, given a variety of common subsets of them -- a result that seems to be common-sense in the water community but has never been rigorously proven. We show that previous existence results are special cases of our more general findings, and therefore lay the foundation for further analysis and theoretical guarantees of the before-mentioned hydraulic simulators.
\end{abstract}




\setcounter{section}{0} 

\setcounter{page}{1} 
\pagenumbering{arabic}



\section{Introduction}
\label{section_Introduction}

The correct estimation of hydraulic states in a \gls{WDS} is a crucial requirement in several tasks and challenges related to such network \cite{Tshehla2017WaterDomain_WaterHydraulics_Overviews}. 
For example, for planning and extension of a \gls{WDS}, \textit{water demands} are of interests. These are connected to \textit{water flows} through every pipe and \textit{hydraulic heads} at every node in the network. The relation between theses hydraulic states are governed by hydraulic principles \cite{Todini2013WaterDomain_StateSimulation_UnifiedFramework,Rossman2020WaterDomain_StateSimulation_Epanet2.2}. 

Since \glspl{WDS} underlie a high security standard and are hardly accessible, experimenting with the system directly is not possible. Instead, scientist and engineers rely on hydraulic simulators, which input a subset of the hydraulic states at a given time and output an estimate of the remaining hydraulic states.
The most prominent hydraulic simulator is EPANET \cite{Rossman2020WaterDomain_StateSimulation_Epanet2.2}, a numeric solver that estimates heads and flows in the whole \gls{WDS} based on heads at reservoirs and demands at all (consumer) nodes.
A more recent emulator is based on the advances of physics-informed \gls{ML}: \citet{Ashraf2024WaterDomain_WaterHydraulics_StateEstimationArchitechtures_StateEstimation} propose a \gls{PI-GCN} as a replacement for EPANET.

When analyzing such simulators from a theoretical perspective, a crucial question is whether theoretically, the input to the simulator, i.e., the subset of observed states, suffices to determine all states of the system in a meaningful way. 
Ideally, in order to derive reliable conclusions from such states, they should be uniquely determined by the subset of known states required as an input. Moreover, the outputted states should obey the before-mentioned underlying physics of the system. 
In \glspl{WDS}, these hydraulic principles translate to non-linear equations between heads and flows and to linear equations between flows and demands \cite{Todini2013WaterDomain_StateSimulation_UnifiedFramework,Rossman2020WaterDomain_StateSimulation_Epanet2.2}. In dependence of what subset of states is known, they translate to a \gls{SLE} or a \gls{SNLE}. The above  questions then go hand in hand with the question whether a solution to such system exists and if so, whether it is unique. 

In the realm of proposing the gradient-based algorithm that lies the foundation of EPANET, \citet{Todini1988WaterDomain_WaterHydraulics_Theory_GradientMethod} discuss this question in the setting where reservoir heads and consumer demands are known by transforming the hydraulic equations to a dual Lagrangian \gls{OP} and analyzing the dual problem.
Moreover,
related questions have so far been investigated on under the name of \textit{observability analysis} \cite{Nagar2012WaterDomain_WaterHydraulics_Theory_ObservabilityAnalysis,Diaz2015WaterDomain_WaterHydraulics_Theory_AlgebraicObservabilityAnalysis,Diaz2017WaterDomain_WaterHydraulics_Theory_TopologicalObservabilityAnalysis},
aiming for strategies \enquote{that evaluate[s] if the measurement set available is sufficient to compute the current state of the network} \cite{Diaz2017WaterDomain_WaterHydraulics_Theory_TopologicalObservabilityAnalysis}.
All these works rely on the linearization of the non-linear hydraulic principles using Taylor approximation. The resulting approximated system of equations is a \gls{SLE} determined by the Jacobian of the hydraulic principles expressed as a function of the heads.

While \citet{Nagar2012WaterDomain_WaterHydraulics_Theory_ObservabilityAnalysis} mainly present a numerical method that outputs ellipsoids which can be interpreted as estimated confidence bounds on the heads under measurement and parametric uncertainty, they shortly discuss requirements on the Jacobian appearing from the Taylor approximation for the solvability of the linearized systems of equations. 

\citet{Diaz2015WaterDomain_WaterHydraulics_Theory_AlgebraicObservabilityAnalysis} investigate on identifying \textit{state variables}, i.e., subsets of heads, demand and/or flows, that -- once observed -- allow to determine the remaining states from that subset based on the hydraulic principles.\footnote{
    Note that \citet{Diaz2015WaterDomain_WaterHydraulics_Theory_AlgebraicObservabilityAnalysis} define such a subset as \textit{state variable} -- a naming that we avoid since we in general call all magnitudes such as heads, demands and flows \textit{states} or \textit{hydraulic states}.
}
They present an algebraic algorithm based on Gauss elimination applied to the before-mentioned Jacobian to answer this question.
The Jacobian and therefore, the whole algorithm, depends on the network topology, i.e., the \gls{WDS}, as well as the subset of known states. Moreover, in order to be able to compute the Jacobian, a reasonable guess for all heads is required, which usually requires domain knowledge. 
\citet{Diaz2017WaterDomain_WaterHydraulics_Theory_TopologicalObservabilityAnalysis} extend the methodology of \citet{Diaz2015WaterDomain_WaterHydraulics_Theory_AlgebraicObservabilityAnalysis} by integrating knowledge of pumps and valves in the \gls{WDS}.

In this work, we present a purely theoretical analysis of the non-linear hydraulic principles which is independent on the specific network topology and gives strong guarantees on the existence and uniqueness of physically correct hydraulic states given a variety of common subsets of them. 
Our analysis neither relies on the linearization of the non-linearity nor requires the execution of any (numerical or algebraic) algorithm.
While there exists work on the setting where reservoir heads and consumer demands are known (comparable to the results of our Theorem \ref{theorem_ExistenceAndUniquenessOfHydraulicStates_ReservoirHeadsAndConsumerDemands}) based on the Lagrangian dual of the primal problem, we provide a more general analysis of multiple settings while avoiding a dual formalization.
Therefore, our work depicts a theoretically founded answer to questions which -- as we noticed during our literature review -- the water community had already considered settled. Moreover, it constitutes a possibility to further theoretical guarantees of hydraulic simulators that are affected directly by the question of the existence of a unique solution to the task they try to solve numerically.

Our work is structured as follows:
In Section \ref{section_WaterDistributionSystemsAsSymmetricDirectedGraphs}, we introduce \glspl{WDS} formally. 
Consequently, in Section \ref{section_WaterHydraulics}, we investigate on water hydraulics, i.e., the hydraulic principles which define physically correct hydraulic states (Section \ref{subsection_WaterHydraulics_HydraulicPrinciples}), and on the existence and uniqueness of such, given a subset of them (Section \ref{subsection_ExistenceAndUniquenessOfHydraulicStates}).
Finally, in Section \ref{section_Conclusion}, we summarize our findings.

\section{Water Distribution Systems as Symmetric Directed Graphs}
\label{section_WaterDistributionSystemsAsSymmetricDirectedGraphs}

A \gls{WDS} is a network that distributes water through \textit{pipes} from water sources (\textit{reservoirs}) to \textit{consumers}. 
This can be modeled as a heterogeneous graph.

\begin{definition}[Water distribution system]
\label{definition_WaterDistributionSystem_SymmetricDirected}
    A \gls{WDS} is a finite, connected and symmetric directed\footnote{
        That is, if $e_{vu} \in E$ holds, then $e_{uv} \in E$ holds as well.
    } 
    graph $\G = (V,E)$ consisting of nodes $V = \{v_1,\dots,v_\n{n}\} \neq \emptyset$ and edges $E = \{e_1,\dots,e_{\n{e}}\} \neq \emptyset$, in which the nodes $V = \Vr \sqcup \Vc$ can be separated into $\n{r} \in \N$ reservoir nodes $\Vr \neq \emptyset$ that supply water to the $\n{c} \in \N$ consumer nodes $\Vc = V \setminus \Vr \neq \emptyset$ through the pipes $E$.
\end{definition}

\noindent 
Reservoirs correspond to the water sources and can be thought as a storage able to provide an infinite amount of water to the system, such as a lake, while consumer junctions may draw water from the network, determined by its demand. The demand not only allows a formal distinction between reservoir nodes $\Vr$ and consumer nodes $\Vc$, but is one of the magnitudes that determines the \textit{hydraulic state} of a \gls{WDS}. 
Next to the \textit{water quality}, the \textit{water hydraulics} is one of the two main dynamics in a \gls{WDS} and the one we will analyze theoretically in the next section.

\section{Water Hydraulics}
\label{section_WaterHydraulics}

Next to its nodes and egdes, a \gls{WDS} is characterized by its \textit{hydraulic state}.

\begin{definition}[Hydraulic states]
\label{definition_HydraulicStates}
    The hydraulic states of a \gls{WDS} $\G = (V,E)$ at time $t \in \R$ corresponds to a triple $(\hm(t),\qm(t),\dm(t))$ consisting of 

    \begin{itemize}
        \item hydraulic heads, or simply \textit{heads}, $\mathbf{h}(t) = (h_v(t))_{v \in V} \in \R_{\geq 0}^{\n{n}}$ at every node,

        \item water flows, or simply \textit{flows}, $\mathbf{q}(t) = (q_e(t))_{e \in E} \in \R^\n{e}$ through every edge and
        
        \item water demands, or simply \textit{demands}, $\mathbf{d}(t) = (d_v(t))_{v \in \Vc} \in \R^{\n{c}}$ at every consumer node.
    \end{itemize}    

    \noindent While we also call single values and subsets of heads, demands and/or flows \textit{hydraulic states}, we sometimes refer to \textit{all} hydraulic states, i.e., the triple $(\hm(t),\qm(t),\dm(t))$, as the \textit{whole hydraulic state}, or in short, just the \textit{hydraulic state}.
\end{definition}

\begin{remark}[Convention of flow signs]
\label{remark_ConventionOfFlowSigns}
    For an edge $e_{vu} \in E$ from $v \in V$ to $u \in \Nbh(v)$ and a time $t \in \R$,
    we use the convention that 
    $q_{vu}(t) := q_{e_{vu}}(t) > 0$ or equivalently, $\sgn(q_{vu}(t)) = 1$ holds 
    \gls{iff} 
    the water flow direction aligns with the edge direction, 
    that is, \gls{iff}
    there is a physical flow from $v$ to $u$.
    Accordingly,
    we use the convention that
    $q_{vu}(t) := q_{e_{vu}}(t) < 0$ or equivalently, $\sgn(q_{vu}(t)) = -1$ holds 
    \gls{iff} 
    the water flow direction is in opposite direction to the edge direction, 
    that is, \gls{iff}
    there is a physical flow from $u$ to $v$.
    If there is no flow between $v \in V$ and $u \in \Nbh(v)$, 
    then $q_{vu}(t) = 0$ or equivalently, $\sgn(q_{vu}(t)) = 0$ holds.
\end{remark}

\noindent 
The introduction of the hydraulic states allows the formal definition
of different kinds of neighborhoods:

    


\begin{definition}[Neighborhoods]
\label{definition_Neighborhoods}

    Let $\G = (V,E)$ be a \gls{WDS} with hydraulic state $(\hm(t),\qm(t),\dm(t))$ at time $t \in \R$.
    Let $v \in V$.
    We denote the inflow, outflow and no-flow neighborhood by
    \begin{align*}
        \Nbh_-(v,t) 
        :=&
        \{ u \in \Nbh(v) ~|~ \sgn(q_{vu}(t)) = - 1 
        \}
        \\
        =& 
        \{ u \in \Nbh(v) ~|~ \text{water flows out from $u$ \textit{in}to $v$ at time $t$} \}
        ,
        \\
        \Nbh_+(v,t) 
        :=&
        \{ u \in \Nbh(v) ~|~ \sgn(q_{vu}(t)) = + 1 
        \}
        \\
        =& 
        \{ u \in \Nbh(v) ~|~ \text{water flows \textit{out} from $v$ into $u$ at time $t$} \}
        \text{ and}
        \\
        \Nbh_0(v,t) 
        :=&
        \{ u \in \Nbh(v) ~|~ \sgn(q_{vu}(t)) = 0 
        \\
        =& 
        \{ u \in \Nbh(v) ~|~ \text{no water flows between $v$ and $u$ at time $t$} \},
    \end{align*}

    \noindent 
    respectively.
\end{definition}


\noindent 
In this section, we investigate on the understanding (Section \ref{subsection_WaterHydraulics_HydraulicPrinciples}) as well as existence and uniqueness (Section \ref{subsection_ExistenceAndUniquenessOfHydraulicStates}) of physically correct hydraulic states, which are governed by hydraulics principles.
Intuitively, the hydraulic state of a \gls{WDS} changes in response to changing demands $\dm(t)$ at consumer nodes $\Vc$ and over time $t \in \R$. 
This demand is provided by the reservoirs $\Vr$ through the change of flows $\qm(t)$ along the edges $E$ and over time $t \in \R$. 
The change of flows, in turn, are a response to a change of the heads $\hm_\Vc(t) = (h_v(t))_{v \in \Vc}$ at consumer nodes $\Vc$ and over time $t \in \R$ due to the change of demand.
In contrast, by the huge size and depth of a reservoir (e.g., a lake), heads $\hm_\Vr(t) = (h_v)_{v \in \Vr}$ at reservoir nodes $\Vr$ can be considered as constant over time $t \in \R$.
In either way, the hydraulic principles describe the \textit{steady state system} of these dynamics. This state is time-independent in the sense that as long as none of the hydraulic states above change, none of the others do. Therefore, when only considering the hydraulics of the system, we omit the time dependency for simplicity.


\subsection{Hydraulic Principles}
\label{subsection_WaterHydraulics_HydraulicPrinciples}

In this subsection, we introduce the three hydraulic principles, which are
the conservation of energy (Definition \ref{definition_HydraulicPrinciples_ConservationOfEnergy}),
the conservation of mass (Definition \ref{definition_HydraulicPrinciples_ConservationOfMass}) and
the conservation of flows (Definition \ref{definition_HydraulicPrinciples_ConservationOfFlows}),
formally. 
All together, they connect all hydraulic states of the whole hydraulic state (cf. Definition \ref{definition_HydraulicStates}).

The \textit{conservation of energy} links the heads to the flows. 
Intuitively, water flows along decreasing heads, that is, if two neighboring nodes have two different heads, water flows from the node with the larger head to the node with the lower head. 
If two neighboring nodes have the same head, there is no flow in between them.
As a consequence of friction along pipes in between two nodes, the overall energy which can be inferred from the heads can not be fully transferred from one node to the other. 
The \textit{Hazen-Williams equation} estimates the head loss caused by this phenomenon:

\begin{definition}[Conservation of energy \cite{Todini2013WaterDomain_StateSimulation_UnifiedFramework}]
\label{definition_HydraulicPrinciples_ConservationOfEnergy}
    The change in between hydraulic heads $h_v$ and $h_u$ between two neighboring nodes $v \in V$ and $u \in \Nbh(v)$, also called \textit{head loss}, is linked to the flow $q_{vu} := q_{e_{vu}}$ along the edge $e_{vu} \in E$ between them by
    \begin{align}
    \label{align_HydraulicPrinciples_ConservationOfEnergy}
        h_v - h_u 
        = 
        r_{vu} \sgn(q_{vu}) |q_{vu}|^x
        =
        r_{vu} q_{vu} |q_{vu}|^{x-1},
    \end{align}
    
    \noindent where $x = 1.852$ is the Hazen-Williams constant and
    $
    r_{vu} = 
    10.67 ~ l_{vu} ~ \dia_{vu}^{-4.8704} ~ \eta_{vu}^{-1.852}
    > 0
    $
    %
    is the pipe-dependent resistance coefficient with symmetric pipe features length $l_{vu} = l_{uv} > 0$, diameter $\dia_{vu} = \dia_{uv} > 0$ and roughness coefficient $\eta_{vu} = \eta_{uv} > 0$.    
\end{definition}

\begin{remark}[Convention of flow signs aligns with physics]
\label{remark_ConventionOfFlowSignsAlignsWithPhysics}
    Since from a physical perspective, water flows along decreasing heads, 
    water \text{outflows} (\text{inflows}) from (to) a node $v \in V$ to (from) a neighboring node $u \in \Nbh(v)$ \gls{iff} $h_v > h_u$ ($h_v < h_u$) holds. 
    Consequently, since by Equation \eqref{align_HydraulicPrinciples_ConservationOfEnergy}, 
    $\sgn(h_v - h_u) = \sgn(q_{vu})$ holds (cf. proof of Theorem \ref{theorem_ExistenceAndUniquenessOfHydraulicStates_ReservoirHeadsAndConsumerHeads}), 
    $q_{vu}$ is an outflow 
    \gls{iff}
    $q_{vu}$ has a positive sign $\sgn(q_{vu}) = \sgn(h_v - h_u) = 1$.
    Analogously, 
    $q_{vu}$ is an inflow
    \gls{iff}
    $q_{vu}$ has a negative sign $\sgn(q_{vu}) = \sgn(h_v - h_u) = - 1$.
    Finally, 
    if there is no flow, that is, if $h_v = h_u$ holds, $q_{vu}$ must be zero with $\sgn(q_{vu}) = \sgn(h_v - h_u) = 0$.  
    Therefore, the convention of flow signs introduced in Remark \ref{remark_ConventionOfFlowSigns} align with the physics.
\end{remark}

\noindent 
The \textit{conservation of mass} links the flows to the demands. 
Intuitively, the amount of the net inflow to a node (that is, the inflow minus its outflow) needs to compensate for the demand at that node:

\begin{definition}[Conservation of mass \cite{Todini2013WaterDomain_StateSimulation_UnifiedFramework}]
\label{definition_HydraulicPrinciples_ConservationOfMass}
    The sum of absolute inflows to a node $v \in \Vc$ 
    minus
    the sum of the outflows from that node
    is equal to
    the demand $d_v$ at that node, 
    or, equivalently,
    the sum of flows at a node $v \in \Vc$ 
    is -- up to the sign -- equal to
    the demand $d_v$ at that node:
    \begin{align}
    \label{align_HydraulicPrinciples_ConservationOfMass}
         \sum_{u \in \Nbh_-(v)} 
         \underbrace{(-q_{vu})}_{> 0} 
         -
         \sum_{u \in \Nbh_+(v)} 
         \underbrace{q_{vu}}_{> 0} 
         = 
         d_v
        \quad \iff \quad
        \sum_{u \in \mathcal{N}(v)} q_{vu} 
        = 
        - d_v.
    \end{align}
\end{definition}

\noindent The \textit{conservation of flows} links the two flows of the two edges that connect the same two nodes, but which show in opposite directions (and which exist since we model the \gls{WDS} as a symmetric directed graph):

\begin{definition}[Conservation of flows \cite{Todini2013WaterDomain_StateSimulation_UnifiedFramework}]
\label{definition_HydraulicPrinciples_ConservationOfFlows}
    For two neighboring nodes $v \in V$ and $u \in \Nbh(v)$, the flow $q_{uv}$ from node $u$ to $v$ is equal to the negative of the flow $q_{vu}$ from $v$ to $u$:
    \begin{align}
    \label{align_HydraulicPrinciples_ConservationOfFlows}
        q_{uv} = - q_{vu}.
    \end{align}
\end{definition}

\noindent
The conservation of flows allows to rewrite the inflow, outflow and no-flow neighborhoods.

\begin{corollary}(Neighborhoods)
\label{corollary_Neighborhoods}
    Let $\G = (V,E)$ be a \gls{WDS} with hydraulic state $(\hm,\qm,\dm)$ that satisfies the conservation of flows (Definition \ref{definition_HydraulicPrinciples_ConservationOfFlows}).
    Let $v \in V$.
    Then the inflow, outflow and no-flow neighborhood by are given by
    \begin{align*}
        \Nbh_\pm(v) 
        :=&
        \{ u \in \Nbh(v) ~|~ \sgn(q_{vu}) = \pm 1 \}
        =
        \{ u \in \Nbh(v) ~|~ \sgn(q_{uv}) = \mp 1 \}
        \text{ and}
        \\
        \Nbh_0(v) 
        :=&
        \{ u \in \Nbh(v) ~|~ \sgn(q_{vu}) = {\color{white} \pm} 0 \}
        =
        \{ u \in \Nbh(v) ~|~ \sgn(q_{uv}) = {\color{white} \pm} 0 \}
        .
    \end{align*}
\end{corollary}

\begin{remark}[Convention of demand signs]
    The conservation of mass (Definition \ref{definition_HydraulicPrinciples_ConservationOfMass}) emphasizes that 
    for a node $v \in V$, the demand $d_v$ is nothing else but another flow drawn from or injected to the system. 
    Together with the conservation of flows (Definition \ref{definition_HydraulicPrinciples_ConservationOfFlows}), 
    Equation \eqref{align_HydraulicPrinciples_ConservationOfMass} becomes
    \begin{align*}
        \underbrace{
        \sum_{u \in \Nbh_-(v)} 
        \underbrace{q_{uv}}_{> 0} 
        }_{
        \text{(in)flow to $v$}
        }
        -
        \underbrace{
        \sum_{u \in \Nbh_+(v)} 
        \underbrace{q_{vu}}_{> 0} 
        }_{
        \text{(out)flow from $v$}
        }
        = 
        d_v.
    \end{align*}

    \noindent
    It emphasizes that 
    $d_v > 0$ holds \gls{iff} the (pipe) inflow to $v$ is larger than the (pipe) outflow from $v$.
    In this case, the node $v$ draws water from the system and is another external outflow from $v$.
    Analogously,
    $d_v < 0$ holds \gls{iff} the (pipe) inflow to $v$ is smaller than the (pipe) outflow from $v$.
    In this case, the node $v$ supplies water to the system and is another external inflow to $v$.
    Both cases align with both the convention of flow signs (Remark \ref{remark_ConventionOfFlowSigns}) and the physics (Remark \ref{remark_ConventionOfFlowSignsAlignsWithPhysics}).
\end{remark}

\noindent 
Finally, given the definition of the three hydraulic principles, we want to call a triple $(\hm,\qm,\dm)$ of hydraulic states physically correct if they obey these principles.

\begin{definition}[Physically correct hydraulic states]
\label{definition_PhysicallyCorrectHydraulicStates}
    The hydraulic states $(\hm,\qm,\dm)$ of a \gls{WDS} $\G = (V,E)$ are called \textit{physically correct} \gls{iff} they satisfy 
    the conservation of energy (Definition \ref{definition_HydraulicPrinciples_ConservationOfEnergy}), 
    the conservation of mass (Definition \ref{definition_HydraulicPrinciples_ConservationOfMass}) and
    the conservation of flows (Definition \ref{definition_HydraulicPrinciples_ConservationOfFlows}),
    i.e.,
    the properties \eqref{align_HydraulicPrinciples_ConservationOfEnergy} to \eqref{align_HydraulicPrinciples_ConservationOfFlows}.
\end{definition}

\noindent In preparation of the analysis of the existence and uniqueness of physically correct hydraulic states given a subset of them, we will now collect some useful observations.

\begin{lemma}[Redundancy of conservation of flows]
\label{lemma_RedundancyOfHydraulicPrinciples_ConservationOfFlows}
    The hydraulic states $(\hm,\qm,\dm)$ of a \gls{WDS} $\G = (V,E)$ are \textit{physically correct} \gls{iff} they satisfy 
    the conservation of energy (Definition \ref{definition_HydraulicPrinciples_ConservationOfEnergy}) and
    the conservation of mass (Definition \ref{definition_HydraulicPrinciples_ConservationOfMass}),
    i.e.,
    the properties \eqref{align_HydraulicPrinciples_ConservationOfEnergy} and \eqref{align_HydraulicPrinciples_ConservationOfMass}.
\end{lemma}

\begin{remark}[Reduncancy of conservation of flows]
\label{remark_RedundancyOfHydraulicPrinciples_ConservationOfFlows}
    A view in the proof of Lemma \ref{lemma_RedundancyOfHydraulicPrinciples_ConservationOfFlows} shows that it only holds if we model the \gls{WDS} as a \textit{symmetric directed} graph.
    In some scenarios, it helps to model the \gls{WDS} as an \textit{oriented} graph. In this case, the conservation of flows in not redundant anymore.
\end{remark}

\noindent In the next theorem, we derive the matrix formulation of the hydraulic principles. 
Hereby, we will stick to the convention of sub-scripting a set to a vector or a matrix if the vector or matrix is limited to that set.

\begin{theorem}[Matrix formulation of hydraulic principles]
\label{theorem_HydraulicPrinciples_MatrixFormulation}
    The hydraulic states $(\hm,\qm,\dm)$ of a \gls{WDS} $\G = (V,E)$ are physically correct \gls{iff}
    \begin{align}
    \label{align_HydraulicPrinciples_ConservationOfEnergy_MatrixForm}
        \Bm^T \hm &= \Dm \qm, \\
    \label{align_HydraulicPrinciples_ConservationOfMass_MatrixForm}
        \Bm_{\Vc \cdot} \qm &= - 2\dm
    \end{align}

    \noindent holds, where $\Bm$ is the incidence matrix
    \begin{align*}
        \Bm 
        :=
        \left(
        \1_{\{ \exists u \in V: ~ e = e_{vu} \}} - \1_{\{ \exists u \in V: ~ e = e_{uv} \}}
        \right)_{ \substack{v \in V \\ e \in E} }
        \in 
        \R^{\n{n} \times \n{e}}
        \text{ of } \G,
    \end{align*}

    \noindent $\Bm_{\Vc \cdot}$ is the incidence matrix $\Bm$ limited to only the rows corresponding to the consumer nodes $\Vc$
    and $\Dm = \Dm(\qm)$ is the flow-dependent diagonal matrix of nonlinearities
    \begin{align*}
        \Dm
        =
        \left(
        \delta_{e_1 e_2} \cdot r_{e_1} |q_{e_1}|^{x-1}
        \right)_{ \substack{e_1 \in E \\ e_2 \in E} }
    \end{align*}

    \noindent 
    with $\delta$ the Kronecker delta, $x$ the Hazen-Williams constant and $r_e$ and $q_e$ the resistance coefficient of and the flow along the edge $e \in E$, respectively (cf. Definition \ref{definition_HydraulicPrinciples_ConservationOfEnergy}).
\end{theorem}

\noindent The following theorem does not apply to \glspl{WDS} only, but to general finite, connected and directed graphs. 
Although we were expecting that the result must have been a well-establish (and proven) fact, to the best of our knowledge, we have not found any literature that indeed proves the following statement. 

\begin{theorem}[Rank of submatrices of the incidence matrix]
\label{theorem_RankOfSubmatricesOfTheIncidenceMatrix}
    Let $\G = (V,E)$ be a finite, connected and directed graph.
    Let $S \subsetneq V$ be a non-empty subset of nodes. 
    Let $\Bm \in \R^{\n{n} \times \n{e}}$ be the incidence matrix (as defined in Theorem \ref{theorem_HydraulicPrinciples_MatrixFormulation}).
    Then $\rk(\Bm_{S \cdot}) = |S|$ holds.
\end{theorem}

\noindent In view of Theorem \ref{theorem_HydraulicPrinciples_MatrixFormulation}, one can already guess that later on, we will be interested into applying Theorem \ref{theorem_RankOfSubmatricesOfTheIncidenceMatrix} for $S = \Vc$.

\begin{remark}[Decomposition of the incidence matrix]
\label{remark_DecompositionOfTheIncidenceMatrix}
    Theorem \ref{theorem_RankOfSubmatricesOfTheIncidenceMatrix} allows a formal proof to the fact that the submatrix $\Bm_{\Vc \cdot}$ ($-A_{21}$ in \citet{Todini2013WaterDomain_StateSimulation_UnifiedFramework}) has full rank of $\rk(\Bm_{\Vc \cdot}) = \n{c}$ ($n$ in \citet{Todini2013WaterDomain_StateSimulation_UnifiedFramework}) and that therefore, it is decomposable into two matrices $\Bm_{\Vc \cdot} = [\Bm_{\Vc \Ei}, \Bm_{\Vc \Ed}] \in \R^{\n{c} \times \n{e}}$ where $\Bm_{\Vc \Ei} \in \R^{\n{c} \times \n{c}}$ is an invertible matrix and $\Bm_{\Vc \Ed} \in \R^{\n{c} \times (\n{e} - \n{c})}$ ($-A_{2n}$ and $-A_{2l}$ in \cite{Todini2013WaterDomain_StateSimulation_UnifiedFramework}, respectively). This legitimates the existence of the invertible matrix $A_{2n}$ in the work of \citet{Todini2013WaterDomain_StateSimulation_UnifiedFramework}.
\end{remark}

\noindent 
We can summarize the findings in Remark \ref{remark_DecompositionOfTheIncidenceMatrix} formally into the following corollary.

\begin{corollary}[Decomposition of the incidence matrix]
\label{corollary_DecompositionOfTheInciddenceMatrix}
    Let $\G = (V,E)$ be a finite, connected and directed graph.
    Let $S \subsetneq V$ be a non-empty subset of nodes. 
    Let $\Bm \in \R^{\n{n} \times \n{e}}$ be the incidence matrix (as defined in Theorem \ref{theorem_HydraulicPrinciples_MatrixFormulation}).
    Then there exist two subsets $\Ei \subset E$ and $\Ed \subset E$ of edges, such that
    
    \begin{enumerate}
        \item $|\Ei| = |S| =: \n{s}$, $|\Ed| = |E| - |S| = \n{e} - \n{s}$ and $E = \Ei \sqcup \Ed$ hold,
        
        \item $\rk(\Bm_{S \cdot}) = \rk(\Bm_{S \Ei}) = \n{s}$ holds,
        
        \item $\Bm_{S \Ei} \in \R^{\n{s} \times \n{s}}$ is invertible and

        \item $\rk(\Bm_{\cdot \Ei}) = \n{s}$ holds.
    \end{enumerate} 

    \noindent 
    More precisely, 
    $\Ei$ is a subset of edges $E$ 
    whose corresponding columns $\Bm_{S \Ei}$ 
    in the limited incidence matrix $\Bm_{S \cdot}$ 
    are (a possible choice of) $\n{s} = |S|$ linearly \textbf{i}ndependent ones; 
    $\Ed = E \setminus \Ei$ is the subset of edges $E$
    whose corresponding columns $\Bm_{S \Ed}$ 
    in the incidence matrix $\Bm_{S \cdot}$ are the $\n{e} - \n{s} = |E| - |S|$ remaining linearly \textbf{d}ependent ones. 
    Even more, 
    $\Ei$ is also a subset of edges $E$
    whose corresponding columns $\Bm_{\cdot \Ei}$
    in the incidence matrix $\Bm$
    are $\n{s}$ of at most $\n{n} - 1$ linearly independent ones.
\end{corollary}

\noindent We have now collected some useful results for the imminent analysis: 
An important task in \glspl{WDS} is to estimate all physically correct state $(\hm,\qm,\dm)$ of a \gls{WDS} given a subset of them. This subset can also be considered as side constraints.

\begin{definition}[Physically correct hydraulic states under side constraints]
\label{definition_PhysicallyCorrectHydraulicStatesUnderSideConstraints}
    Let $\G = (V,E)$ be a \gls{WDS} and $\hm_\Vh^*, \qm_\Eq^*$ and $\dm_\Vd^*$ a subset of (known) hydraulic states on subsets $\Vh\subset V$, $\Eq \subset E$ and $\Vd \subset \Vc$, respectively.
    \\A set of hydraulic states $(\hm,\qm,\dm)$ is called \textit{physically correct \gls{wrt} these side constraints} \gls{iff} they are physically correct and satisfy the side constraints, i.e.,
    \begin{align*}
        h_v &= h_v^* \text{ for all } v \in \Vh \quad \text{and} \\
        q_e &= q_e^* \text{ for all } e \in \Eq \quad \text{and} \\
        d_v &= d_v^* \text{ for all } v \in \Vd \quad \text{hold}.\\
    \end{align*}
\end{definition}

\noindent For example, the well-known hydraulic simulator EPANET \cite{Rossman2020WaterDomain_StateSimulation_Epanet2.2}, but also more recent physics-informed and \gls{ML}-based surrogate models \cite{Ashraf2024WaterDomain_WaterHydraulics_StateEstimationArchitechtures_StateEstimation} assume to have knowledge of reservoir heads $\hm_\Vr = (h_v)_{v \in \Vr}$ and consumer demands $\dm_\Vc = \dm = (d_v)_{v \in \Vc}$ in order to estimate consumer heads $\hm_\Vc = (h_v)_{v \in \Vc}$ and flows $\qm_E = \qm = (q_e)_{e \in E}$.
If the overall estimation $(\hm,\qm,\dm) := ([\hm_\Vr^T,\hm_\Vc^T]^T,\qm,\dm)$ is physically correct, it is automatically physically correct \gls{wrt} the reservoir heads $\hm_\Vr^* = \hm_\Vr$ and the consumer demands $\dm_\Vc^* = \dm_\Vc$ (i.e., in terms of Definition \ref{definition_PhysicallyCorrectHydraulicStatesUnderSideConstraints}, we have $\Vh = \Vr$, $\Eq = \emptyset$ and $\Vd = \Vc$ in this example). 

More generally, a natural question is for which subset of known hydraulic states $\hm_\Vh, \qm_\Eq$ and $\dm_\Vd$, 
there exist remaining hydraulic states $\hm_{V \setminus \Vh}, \qm_{E \setminus \Eq}$ and $\dm_{\Vc \setminus \Vd}$, 
such that the triple $(\hm,\qm,\dm) := ([\hm_\Vh^T,\hm_{V \setminus \Vh}^T]^T,[\qm_\Eq^T,\qm_{E \setminus \Eq}^T]^T,[\dm_\Vd^T,\dm_{\Vc \setminus \Vd}^T]^T)$ is physically correct \gls{wrt} these side constraints (i.e., \gls{wrt} $\hm_\Vh^* = \hm_\Vh, \qm_\Eq^* = \qm_\Eq$ and $\dm_\Vd^* = \dm_\Vd$). Even more, it is of interest whether the remaining hydraulic states are uniquely determined by the known ones.
We will investigate on this question in the following Section \ref{subsection_ExistenceAndUniquenessOfHydraulicStates}.

\begin{remark}[Existence of physically correct hydraulic states under side constraints from the real world]
\label{remark_ExistenceOfPhysicallyCorrectHydraulicStatesUnderSideConstraintsFromTheRealWorld}
    If the subset of known hydraulic states corresponds to error-free observations from the real world, there trivially exists remaining hydraulic states as investigated above.
    However, in practice, observations from the real world are prone to errors, for example, due to measurement inaccuracies. Therefore, we are generally interested into the existence and uniqueness of physically correct hydraulic states \gls{wrt} \textit{general} side constraints, i.e., side constraints that might not correspond to error-free observations from the real world. 
\end{remark}

\subsection{Existence and Uniqueness of Hydraulic States}
\label{subsection_ExistenceAndUniquenessOfHydraulicStates}

By Theorem \ref{theorem_HydraulicPrinciples_MatrixFormulation}, physically correct hydraulic states $(\hm,\qm,\dm)$ are determined by the Equations \eqref{align_HydraulicPrinciples_ConservationOfEnergy_MatrixForm} and \eqref{align_HydraulicPrinciples_ConservationOfMass_MatrixForm}, which correspond to the conservation of energy (Definition \ref{definition_HydraulicPrinciples_ConservationOfEnergy}) and the conservation of mass (Definition \ref{definition_HydraulicPrinciples_ConservationOfMass}), respectively:
\begin{align*}
    \Bm^T \hm &= \Dm \qm, \\
    \Bm_{\Vc \cdot} \qm &= - 2\dm.
\end{align*}

\noindent In dependence of what subset of states is known, they translate to a \gls{SLE} or a \gls{SNLE}, which we can analyze \gls{wrt} the question of the existence and uniqueness of a solution.\footnote{
    Note that if not stated otherwise, we consider the system of equations as such over the vector space $\R^n$, where the dimension $n \in \N$ depends on the states known and the states unknown. 
    In the context of \glspl{WDS}, the additional assumption that the heads $\hm \in \R_{\geq 0}^\n{n}$ are non-negative (cf. Definition \ref{definition_HydraulicStates}) is only a consequence of the physical context. If we find a solution to such a system of equations which includes heads and which is based on known states from the real world, i.e., based on know states that obey these physics, the additional assumptions will automatically be satisfied.
    Also note that accordingly, the following theorems are purely mathematical results that also hold if the known states are \textit{not} observations from the real world (cf. Remark \ref{remark_ExistenceOfPhysicallyCorrectHydraulicStatesUnderSideConstraintsFromTheRealWorld}). We will see in Theorem \ref{theorem_ExistenceAndUniquenessOfHydraulicStates_ReservoirHeadsAndFlows} and Lemma \ref{lemma_ExistenceAndUniquenessOfHydraulicStates_OnReservoirHeadsAndFlowsCondition} that the \textit{additional} assumption that the known states are observations from the real world can cause the removal of another technical assumption required in the purely mathematical context.
}

\begin{theorem}[Reservoir heads and consumer heads determine the hydraulic state uniquely]
\label{theorem_ExistenceAndUniquenessOfHydraulicStates_ReservoirHeadsAndConsumerHeads}
    Let $\G = (V,E)$ be a \gls{WDS} and 
    $\hm = \hm_V = [\hm_\Vr^T,\hm_\Vc^T]^T$ 
    known 
    (reservoir and consumer) heads.
    \\There exists exactly one 
    tuple $(\qm,\dm)$ 
    of flows and demands
    such that the triple 
    $(\hm,\qm,\dm)$ 
    of hydraulic states is physically correct \gls{wrt} these side constraints.
\end{theorem}

\begin{remark}[Theorem \ref{theorem_ExistenceAndUniquenessOfHydraulicStates_ReservoirHeadsAndConsumerHeads} lies the foundation for the work of \citet{Nagar2012WaterDomain_WaterHydraulics_Theory_ObservabilityAnalysis} and \citet{Diaz2015WaterDomain_WaterHydraulics_Theory_AlgebraicObservabilityAnalysis}]
    As elaborated in the Section \ref{section_Introduction}, the work of \citet{Nagar2012WaterDomain_WaterHydraulics_Theory_ObservabilityAnalysis} and \citet{Diaz2015WaterDomain_WaterHydraulics_Theory_AlgebraicObservabilityAnalysis} rely on the linearization of the non-linear hydraulic principles using Taylor approximation. The resulting approximated system of equations is a \gls{SLE} determined by the Jacobian of the hydraulic principles expressed as a function of the heads $\hm$.
   This is based on the assumption that the heads $\hm$ display what the authors define as \textit{state variables}\footnote{
        We avoid the usage of the term \textit{state variable} since we in general call all magnitudes such as heads, demands and flows \textit{states} or \textit{hydraulic states}.
   }, i.e., \enquote{a minimal set of variables whose values are sufficient to compute, by using the network model, the value of any other network variable}\footnote{
        In this context, the network model corresponds to \enquote{the nonlinear relationship [...] derived from the network's hydraulic model}. 
        The authors refer to the explicit hydraulic principles as presented in Definition \ref{definition_HydraulicPrinciples_ConservationOfEnergy} and \ref{definition_HydraulicPrinciples_ConservationOfMass}.
    } \cite{Diaz2015WaterDomain_WaterHydraulics_Theory_AlgebraicObservabilityAnalysis}.
    In other words, the approaches are based on the assumption that by the hydraulic principles as introduced in Section \ref{subsection_WaterHydraulics_HydraulicPrinciples}, the whole hydraulic state $(\hm,\qm,\dm)$ can be inferred from the heads $\hm$. 
    This is exactly what Theorem \ref{theorem_ExistenceAndUniquenessOfHydraulicStates_ReservoirHeadsAndConsumerHeads} states.
    Even more, it states that the whole hydraulic state $(\hm,\qm,\dm)$ is uniquely determined by the heads $\hm$.

    Additionally, this work gives theoretically rigorous answers to the question which subsets of hydraulic states beyond the heads define what the works \citet{Nagar2012WaterDomain_WaterHydraulics_Theory_ObservabilityAnalysis,Diaz2015WaterDomain_WaterHydraulics_Theory_AlgebraicObservabilityAnalysis,Diaz2017WaterDomain_WaterHydraulics_Theory_TopologicalObservabilityAnalysis} define as a state variable.
    This lies the foundation to possible adaptations of their approaches.
\end{remark}

\begin{remark}[Heads imply flows and flows imply demands]
\label{remark_HeadsImplyFlowsAndFlowsImplyDemands}
    In the proof of Theorem \ref{theorem_ExistenceAndUniquenessOfHydraulicStates_ReservoirHeadsAndConsumerHeads}, we show that 
    heads $\hm$ uniquely determine flows $\qm$ 
    and 
    flows $\qm$ uniquely determine demands $\dm$.
    The opposite directions do not hold:
    Given flows $\qm$, there can be none or multiple heads $\hm$ that satisfy $\Bm^T \hm = \Dm \qm$, and given demands $\dm$, there can be multiple flows $\qm$ that satisfy $\Bm_{\Vc \cdot} \qm = - 2 \dm$.\footnote{
        A common result in graph theory is that $\rk(\Bm^T) = \rk(\Bm) = \n{n} - 1$ holds (cf. \citet{Bapat2014Graphtheory}).
        Moreover, since $\G$ is a connected graph, $\n{n} - 1 \leq \n{e}$ holds.
        Thus, by basics results from linear algebra, the matrix $\Bm^T$ induces a linear map $\R^\n{n} \rightarrow \R^\n{e}, \hm \mapsto \Bm^T \hm$ which is neither surjective nor injective.
        Similarly, by Theorem \ref{theorem_RankOfSubmatricesOfTheIncidenceMatrix}, $\rk(\Bm_{\Vc \cdot}) = \n{c}$ holds. 
        Therefore, as we discuss in Remark \ref{remark_ExistenceAndUniquenessOfHydraulicStates_ReservoirHeadsAndFlows} or \ref{remark_OnTheConservationOfMass}, 
        $\n{c} \leq \n{e}$ holds and the matrix $\Bm_{\Vc \cdot}$ induces a linear map $\R^\n{c} \rightarrow \R^\n{e}, \qm \mapsto \Bm_{\Vc \cdot} \qm$ which is bijective if $\n{c} = \n{e}$ holds, but only surjective and not injective in the more common case of $\n{c} < \n{e}$.
    }
\end{remark}

\noindent 
Usually, 
instead of considering all heads $\hm \in \R^\n{n} = \R^{\n{r} + \n{c}}$, 
one distinguishes between heads $\hm_\Vr \in \R^\n{r}$ at reservoirs and heads $\hm_\Vc \in \R^{\n{c}}$ at consumer nodes. 
Based on this distinction, similar to \citet{Todini2013WaterDomain_StateSimulation_UnifiedFramework} and after eventually resorting, we can rewrite the corresponding vectors, matrices and multiplications as
\begin{align*}
    \hm 
    =
    \begin{bmatrix}
        \hm_\Vr \\
        \hm_\Vc
    \end{bmatrix}
    ,~
    \Bm 
    =
    \begin{bmatrix}
        \Bm_{\Vr \cdot} \\
        \Bm_{\Vc \cdot}
    \end{bmatrix}
    ,~
    \Bm^T \hm
    =
    \begin{bmatrix}
        \Bm_{\Vr \cdot}^T & \Bm_{\Vc \cdot}^T
    \end{bmatrix}
    \cdot 
    \begin{bmatrix}
        \hm_\Vr \\
        \hm_\Vc
    \end{bmatrix}
    =
    \Bm_{\Vr \cdot}^T \hm_\Vr + \Bm_{\Vc \cdot}^T \hm_\Vc
\end{align*}

\noindent 
and 
by this, Equation \eqref{align_HydraulicPrinciples_ConservationOfEnergy_MatrixForm} and \eqref{align_HydraulicPrinciples_ConservationOfMass_MatrixForm} translate to 
\begin{align}
\label{align_HydraulicPrinciples_ConservationOfEnergy_MatrixForm_separatedHeads}
    \Bm_{\Vr \cdot}^T \hm_\Vr + \Bm_{\Vc \cdot}^T \hm_\Vc &= \Dm \qm, \\
\label{align_HydraulicPrinciples_ConservationOfMass_MatrixForm_separatedHeads}
    \Bm_{\Vc \cdot} \qm &= - 2\dm.
\end{align}

\begin{theorem}[Reservoir heads and flows determine the hydraulic state uniquely]
\label{theorem_ExistenceAndUniquenessOfHydraulicStates_ReservoirHeadsAndFlows}
    Let $\G = (V,E)$ be a \gls{WDS} and 
    $\hm_\Vr$ and $\qm = \qm_E$ 
    known 
    reservoir heads and flows, respectively.
    \\
    If $\Dm \qm - \Bm_{\Vr \cdot}^T \hm_\Vr \in \im(\Bm_{\Vc \cdot}^T) \subset \R^{\n{e}}$ holds,
    there exists exactly one 
    tuple $(\hm_\Vc,\dm)$ 
    of consumer heads and demands 
    such that the triple 
    $(\hm,\qm,\dm) = ([\hm_\Vr^T,\hm_\Vc^T]^T,\qm,\dm)$ 
    of hydraulic states is physically correct \gls{wrt} these side constraints.
\end{theorem}

\noindent 
While the following Remark \ref{remark_ExistenceAndUniquenessOfHydraulicStates_ReservoirHeadsAndFlows} investigates on the necessity of the condition $\Dm \qm - \Bm_{\Vr \cdot}^T \hm_\Vr \in \im(\Bm_{\Vc \cdot}^T) \subset \R^{\n{e}}$ in Theorem \ref{theorem_ExistenceAndUniquenessOfHydraulicStates_ReservoirHeadsAndFlows}, 
Lemma \ref{lemma_ExistenceAndUniquenessOfHydraulicStates_OnReservoirHeadsAndFlowsCondition} shows that it is naturally satisfied if the known reservoir heads $\hm_\Vr$ and flows $\qm$ are part of a physically correct hydraulic state.
This is, for example, the case if they are error-free observations from the real world (also cf. Remark \ref{remark_ExistenceOfPhysicallyCorrectHydraulicStatesUnderSideConstraintsFromTheRealWorld}).

\begin{remark}[Reservoir heads and flows determine the hydraulic state uniquely]
\label{remark_ExistenceAndUniquenessOfHydraulicStates_ReservoirHeadsAndFlows}
    Theorem \ref{theorem_ExistenceAndUniquenessOfHydraulicStates_ReservoirHeadsAndFlows} requires the condition that $\Dm \qm - \Bm_{\Vr \cdot}^T \hm_\Vr \in \im(\Bm_{\Vc \cdot}^T) \subset \R^{\n{e}}$ holds, 
    which guarantees the existence of a solution $\hm_\Vc \in \R^\n{c}$ to the \gls{SLE} %
    \begin{align*}
        \Bm_{\Vc \cdot}^T \hm_\Vc = \Dm \qm - \Bm_{\Vr \cdot}^T \hm_\Vr
    \end{align*}

    \noindent and therefore, to Equation \eqref{align_HydraulicPrinciples_ConservationOfEnergy_MatrixForm_separatedHeads}.
    The matrix $\Bm_{\Vc \cdot}^T \in \R^{\n{e} \times \n{c}}$ induces a linear map $\R^\n{c} \longrightarrow \R^\n{e}, \hm_\Vc \longmapsto \Bm_{\Vc \cdot}^T \hm_\Vc$. Since by Theorem \ref{theorem_RankOfSubmatricesOfTheIncidenceMatrix}, $\rk(\Bm_{\Vc \cdot}^T) = \rk(\Bm_{\Vc \cdot}) = \n{c}$ holds, the map is injective, which -- together with the assumption that $\Dm \qm - \Bm_{\Vr \cdot}^T \hm_\Vr \in \im(\Bm_{\Vc \cdot}^T) \subset \R^{\n{e}}$ holds -- makes this solution $\hm_\Vc$ unique. 
    
    Moreover, note that 
    by $\n{c} = \rk(\Bm_{\Vc \cdot}) \leq \min\{\n{c},\n{e}\} \leq \n{c}$, $\min\{\n{c},\n{e}\} = \n{c}$ and thus, $\n{e} \geq \n{c}$ follows. 
    If $\n{e} = \n{c}$ holds, the linear map is not only injective, but bijective, and $\Dm \qm - \Bm_{\Vr \cdot}^T \hm_\Vr \in \im(\Bm_{\Vc \cdot}^T)$ trivially holds.
    However, in the usual case where $\n{e} > \n{c}$ holds, the linear map can not be surjective, making $\Dm \qm - \Bm_{\Vr \cdot}^T \hm_\Vr \in \im(\Bm_{\Vc \cdot}^T)$ a necessary assumption:
    Since by $\rk(\Bm_{\Vc \cdot}^T) = \n{c}$, the solution $\hm_\Vc$ is already determined by $\n{c}$ of the $\n{e}$ equations of the \gls{SLE} above, it can happen that one or multiple of the $\n{e} - \n{c}$ remaining equations contradict the previous $\n{c}$ ones. In this case, $\Dm \qm - \Bm_{\Vr \cdot}^T \hm_\Vr \not \in \im(\Bm_{\Vc \cdot}^T)$ holds.
\end{remark}

\begin{lemma}[Reservoir heads and flows determine the hydraulic state uniquely]
\label{lemma_ExistenceAndUniquenessOfHydraulicStates_OnReservoirHeadsAndFlowsCondition}
    Let $\G = (V,E)$ be a \gls{WDS} and 
    $\hm_\Vr$ and $\qm = \qm_E$ 
    known 
    reservoir heads and flows, respectively.
    If there exist unobserved consumer heads $\hm_\Vc$ and demands $\dm$ such that
    the triple $(\hm,\qm,\dm) = ([\hm_\Vr^T,\hm_\Vc^T]^T,\qm,\dm)$ 
    of hydraulic states is physically correct \gls{wrt} these side constraints,
    then $\Dm \qm - \Bm_{\Vr \cdot}^T \hm_\Vr \in \im(\Bm_{\Vr \cdot}^T)$ holds.
\end{lemma}

\noindent As investigated in Remark \ref{remark_ExistenceAndUniquenessOfHydraulicStates_ReservoirHeadsAndFlows}, the knowledge of reservoir heads $\hm_\Vr \in \R^\n{r}$ and \textit{all} flows $\qm \in \R^\n{e}$ is not a minimal assumption in order to obtain existence and uniqueness of physically correct hydraulic states.
In order to get rid of the additional assumption in Theorem \ref{theorem_ExistenceAndUniquenessOfHydraulicStates_ReservoirHeadsAndFlows}, we can make use of Corallary \ref{corollary_DecompositionOfTheInciddenceMatrix} with $S = \Vc$ to decompose the set of edges $E$ into the subsets $\Ei$ and $\Ed$, corresponding to the $\n{c}$ linearly \textbf{i}ndependent and $\n{e} - \n{c}$ linearly \textbf{d}ependent columns of the limited incidence matrix $\Bm_{\Vc \cdot} \in \R^{\n{c} \times \n{e}}$,  respectively.
Note that according to Corallary \ref{corollary_DecompositionOfTheInciddenceMatrix}.4, $\Ei$ also corresponds to $\n{c}$ linearly \textbf{i}ndependent columns of the incidence matrix $\Bm \in \R^{\n{n} \times \n{e}}$.

Consequently and similar to as we have done with the heads earlier in this subsection,
instead of considering all flows $\qm \in \R^\n{e} = \R^{\n{c} + (\n{e} - \n{c})}$, 
we can distinguish between the flows $\qm_\Ei \in \R^\n{c}$ and the flows $\qm_\Ed \in \R^{\n{e} - \n{c}}$. 
Based on this distinction and after eventually resorting, we can rewrite the corresponding vectors, matrices and multiplications as detailed in the proof of Theorem \ref{theorem_ExistenceAndUniquenessOfHydraulicStates_ReservoirHeadsAndSomeFlows}.
By this, Equation \eqref{align_HydraulicPrinciples_ConservationOfEnergy_MatrixForm_separatedHeads} and \eqref{align_HydraulicPrinciples_ConservationOfMass_MatrixForm_separatedHeads} translate to
\begin{align}
\label{align_HydraulicPrinciples_ConservationOfEnergy_MatrixForm_separatedHeads_separatedFlows_independentFlows}
    \Bm_{\cdot \Ei}^T \hm
    =&~
    \Bm_{\Vr \Ei}^T \hm_\Vr + \Bm_{\Vc \Ei}^T \hm_\Vc 
    ~~=
    \Dm_{\Ei \Ei} ~ \qm_\Ei, \\    
\label{align_HydraulicPrinciples_ConservationOfEnergy_MatrixForm_separatedHeads_separatedFlows_dependentFlows}
    \Bm_{\cdot \Ed}^T \hm
    =&~
    \Bm_{\Vr \Ed}^T \hm_\Vr + \Bm_{\Vc \Ed}^T \hm_\Vc
    ~= 
    \Dm_{\Ed \Ed} \qm_\Ed, \\
\label{align_HydraulicPrinciples_ConservationOfMass_MatrixForm_separatedHeads_separatedFlows}
    \Bm_{\Vc \cdot} \qm
    =&~
    \Bm_{\Vc \Ei} \qm_\Ei + \Bm_{\Vc \Ed} ~ \qm_\Ed 
    = 
    - 2\dm,
\end{align}

\noindent with the peculiarity that according to Corallary \ref{corollary_DecompositionOfTheInciddenceMatrix}.3, $\Bm_{\Vc \Ei} \in \R^{\n{c} \times \n{c}}$, and therefore, also $\Bm_{\Vc \Ei}^T \in \R^{\n{c} \times \n{c}}$ is an invertible matrix.

\begin{theorem}[Reservoir heads and some flows determine the hydraulic state uniquely]
\label{theorem_ExistenceAndUniquenessOfHydraulicStates_ReservoirHeadsAndSomeFlows}
    Let $\G = (V,E)$ be a \gls{WDS} and 
    $\hm_\Vr$ and $\qm_\Ei$ 
    known 
    reservoir heads and flows, respectively,
    where -- as introduced above -- $\Ei$ is any subset of the edges $E$ corresponding to $\n{c}$ linearly independent columns of the limited incidence matrix $\Bm_{\Vc \cdot}$ of $\G$.
    \\There exists exactly one 
    triple $(\hm_\Vc,\qm_\Ed,\dm)$  
    of consumer heads, remaining flows and demands
    such that the triple 
    $(\hm,\qm,\dm) = ([\hm_\Vr^T,\hm_\Vc^T]^T,[\qm_\Ei^T,\qm_\Ed^T]^T,\dm)$ 
    of hydraulic states is physically correct \gls{wrt} these side constraints.
\end{theorem}

\noindent 
It is a common result in graph theory that the $\n{c}$ linear independent columns $\Bm_{\cdot E_i} \in \R^{\n{n} \times \n{c}}$ of the incidence matrix $\Bm \in \R^{\n{n} \times \n{e}}$ induce an acyclic sub-graph \cite{Bapat2014Graphtheory}, 
which in this case connects each of the $\n{c}$ consumer nodes $\Vc$ to exactly one reservoir node $v \in \Vr$ by in total $\n{c}$ edges $\Ei$, without any cycles in betweem them.
The resulting graph structure is also known as a \textit{forest}, where each \textit{tree} has one reservoir included. In the case of one reservoir only, i.e., if $\n{r} = 1$ holds, this breaks down to a simple tree structure.
In view of Equation \eqref{align_HydraulicPrinciples_ConservationOfEnergy_MatrixForm_separatedHeads_separatedFlows_independentFlows}, Theorem \ref{theorem_ExistenceAndUniquenessOfHydraulicStates_ReservoirHeadsAndSomeFlows} states that the flows $\qm_\Ei$ along pipes that define such a forest (tree) are -- in combination with the reservoir heads $\hm_\Vr$ -- sufficient to determine the heads $\hm_\Vc$ at consumer nodes and -- as detailed in the proof of Theorem \ref{theorem_ExistenceAndUniquenessOfHydraulicStates_ReservoirHeadsAndSomeFlows} -- consequently, the whole hydraulic state.
This observation generalizes the results of \citet{Diaz2015WaterDomain_WaterHydraulics_Theory_AlgebraicObservabilityAnalysis}:

\begin{remark}[Theorem \ref{theorem_ExistenceAndUniquenessOfHydraulicStates_ReservoirHeadsAndSomeFlows} generalizes the results of \citet{Diaz2015WaterDomain_WaterHydraulics_Theory_AlgebraicObservabilityAnalysis}]
    
    \citet{Diaz2015WaterDomain_WaterHydraulics_Theory_AlgebraicObservabilityAnalysis} present two examples to which they illustratively apply their proposed algebraic algorithm. In both examples, reservoir heads and a subset of flows are given, and they differ in the choice of that subset.
    In the first example, flows along four out of seven pipes that connect all four consumer nodes to a reservoir node are given, therefore naturally spanning a forest (tree) and thus not including cycles. The result of the algebraic algorithm is that \enquote{the system state is observable} \cite{Diaz2015WaterDomain_WaterHydraulics_Theory_AlgebraicObservabilityAnalysis}, i.e., that one can infer the whole hydraulic state from this subset of hydraulic states.
    In the second example, flows along four out of seven pipes that connect three out of four consumer nodes to a reservoir node are given, therefore naturally including a cycle. The result of the algebraic algorithm is that \enquote{the system is unobservable} \cite{Diaz2015WaterDomain_WaterHydraulics_Theory_AlgebraicObservabilityAnalysis}, i.e., that one cannot infer the whole hydraulic state from this subset of hydraulic states.

    As discussed above, these two examples are the two possible cases of choices of flows: 
    Either, the subset of known flows corresponds to $\n{c}$ linearly independent columns of the limited incidence matrix $\Bm_{\Vc \cdot}$ of $\G$, which corresponds to edges that connect all consumer nodes $\Vc$ to a reservoir without any circles (first example). 
    In this case, the whole hydraulic state can be inferred from this subset together with the reservoir heads according to Theorem \ref{theorem_ExistenceAndUniquenessOfHydraulicStates_ReservoirHeadsAndSomeFlows}.
    Or, the subset of known flows corresponds to linearly dependent columns of the limited incidence matrix $\Bm_{\Vc \cdot}$ of $\G$, which corresponds to edges that connect only some consumer nodes $\Vc$ to a reservoir while introducing cycles (second example).
    In this case, the whole state can not be inferred. 
    Indeed, 
    if $\rk(\Bm_{\Vc \Ei}) < \n{c}$ holds,
    by basic results from linear algebra, the matrix $\Bm_{\Vc \Ei}^T \in \R^{\n{c} \times \n{c}}$ induces a linear map $\R^\n{c} \longrightarrow \R^\n{c}, \hm_\Vc \longmapsto \Bm_{\Vc \Ei}^T \hm_\Vc$ that can neither be injective nor surjective. 
    Therefore, in this case, the \gls{SLE}  $\Bm_{\Vc \Ei}^T \hm_\Vc = \Dm_{\Ei \Ei} \qm_\Ei - \Bm_{\Vr \Ei}^T \hm_\Vr$ (cf. Eq. \eqref{align_HydraulicPrinciples_ConservationOfEnergy_MatrixForm_separatedHeads_separatedFlows_independentFlows}) might not be solvable, and if it is, the solution is not unique.
\end{remark}

\noindent 
The next theorem displays are a very important result since the subset of known hydraulic states corresponds to the ones the well-known hydraulic simulator EPANET \cite{Rossman2020WaterDomain_StateSimulation_Epanet2.2}, but also more recent physics-informed and \gls{ML}-based surrogate models \cite{Ashraf2024WaterDomain_WaterHydraulics_StateEstimationArchitechtures_StateEstimation} assume to have knowledge of.
In order to be able to prove it, an important property of the non-linearity in Equation \eqref{align_HydraulicPrinciples_ConservationOfEnergy_MatrixForm} is that it defines a strictly monotone operator.

\begin{lemma}[The non-linearity defines a strictly monotone operator]
\label{lemma_TheNonLinearityDefinesAStrictilyMonotoneOperator}
    The function
    \begin{align*}
        f: \R^\n{e} \longrightarrow \R^\n{e}, ~
        \qm 
        \longmapsto 
        f(\qm) 
        := 
        \Dm \qm
        =
        \left(
        r_e q_e |q_e|^{x-1}
        \right)_{e \in E}
        =:
        \left(
        f_e(\qm)
        \right)_{e \in E}
    \end{align*}

    \noindent satisfies
    \begin{align*}
        \langle 
        f(\qm_1) - f(\qm_2), \qm_1 - \qm_2
        \rangle
        =
        \sum_{e \in E}
        x ~ r_e ~ (q_{1e} - q_{2e})^2
        \int_{0}^{1}
        |q_{2e} + t (q_{1e} - q_{2e})|^{x-1}
        ~dt
        \geq 
        0
    \end{align*}

    \noindent for all $\qm_1, \qm_2 \in \R^\n{e}$. Even more, it defines a strictly monotone operator, that is, it satisfies
    \begin{align*}
        \langle 
        f(\qm_1) - f(\qm_2), \qm_1 - \qm_2
        \rangle
        >
        0
    \end{align*}

    \noindent for all $\qm_1, \qm_2 \in \R^\n{e}$ with $\qm_1 \neq \qm_2$.
\end{lemma}    

\noindent 
The monotonicity of the non-linearity $f$ is the key ingredient for our final existence and uniqueness result

\begin{theorem}[Reservoir heads and consumer demands determine the hydraulic state uniquely]
\label{theorem_ExistenceAndUniquenessOfHydraulicStates_ReservoirHeadsAndConsumerDemands}
    Let $\G = (V,E)$ be a \gls{WDS} and
    $\hm_\Vr$ and $\dm = \dm_\Vc$ 
    known 
    reservoir heads and (consumer) demands, respectively.
    \\There exists exactly one 
    tuple $(\hm_\Vc,\qm)$ 
    of consumer heads and flows 
    such that the triple 
    $(\hm,\qm,\dm) = ([\hm_\Vr^T,\hm_\Vc^T]^T,\qm,\dm)$ 
    of hydraulic states is physically correct \gls{wrt} these side constraints.
\end{theorem}

\section{Conclusion}
\label{section_Conclusion}

In this work, we conducted an analysis on the existence and uniqueness of physically correct hydraulic states based on a variety of common given subsets of them. In dependence on which subset of states is known, this research question translates to the solvability of a \gls{SLE} or a \gls{SNLE}. 
In contrast to previous works, we did not work with a dual formulation and did not linearize appearing nonlinearities but analyzed the nonlinear hydraulic principles directly. 
Moreover, the results are network independent and purely theoretical. This avoids the necessity of computing any error-prone numerical or algebraic algorithm per choice of a \gls{WDS} \textit{and} subset of hydraulic states, which is linked to further knowledge such as an initial estimation of hydraulic heads.

Besides these benefits as compared to previous, algorithmic-based analyses, we actually proved that the before-mentioned hydraulic heads determine the whole hydraulic state through the hydraulic principles uniquely -- an assumption made and required in some of those previous works. 
Moreover, we identified the results of one work as special cases to one of our more general theorems, emphasizing the powerfulness of our results.

In conclusion, our work depicts a theoretically founded answer to questions which -- as we noticed during our literature review -- the water community has already considered settled. Moreover, it constitutes a possibility to further theoretical guarantees of hydraulic simulators that are affected directly by the question of the existence of a unique solution to the task they try to solve numerically.



\newpage
\setglossarystyle{list}
\printglossary[title=Glossary, nonumberlist]



\setglossarystyle{list}
\printglossary[type=\acronymtype, title=List of Abbreviations, nonumberlist]



\newpage






\newpage
\pagenumbering{Roman}
\appendix

\section*{Organization of the Supplementary Material}
\startcontents[sections]
\printcontents[sections]{l}{1}{\setcounter{tocdepth}{3}}

\section{Proofs}

\subsection{Proof of Lemma \ref{lemma_RedundancyOfHydraulicPrinciples_ConservationOfFlows}}

\begin{lemmaNoNb}[Lemma \ref{lemma_RedundancyOfHydraulicPrinciples_ConservationOfFlows}]
    The hydraulic states $(\hm,\qm,\dm)$ of a \gls{WDS} $\G = (V,E)$ are \textit{physically correct} \gls{iff} they satisfy 
    the conservation of energy (Definition \ref{definition_HydraulicPrinciples_ConservationOfEnergy}) and
    the conservation of mass (Definition \ref{definition_HydraulicPrinciples_ConservationOfMass}),
    i.e.,
    the properties \eqref{align_HydraulicPrinciples_ConservationOfEnergy} and \eqref{align_HydraulicPrinciples_ConservationOfMass}.
\end{lemmaNoNb}

\begin{proof}
    \enquote{$\Rightarrow$}
    If the hydraulic states $(\hm,\qm,\dm)$ are \textit{physically correct}, they trivially satisfy 
    the conservation of energy and
    the conservation of mass.
    \\
    \\
    \enquote{$\Leftarrow$}
    If in contrast, the hydraulic states $(\hm,\qm,\dm)$ satisfy 
    the conservation of energy and
    the conservation of mass, 
    we need to show that they already satisfy 
    the conservation of flows, i.e., 
    property \eqref{align_HydraulicPrinciples_ConservationOfFlows}.
    
    To do so, let $v \in V$ and $u \in \Nbh(v)$ hold. By the definition of the neighborhood, $e_{vu} \in E$ or $e_{uv} \in E$ holds, and since the graph $\G = (V,E)$ is a symmetric directed graph (cf. Definition \ref{definition_WaterDistributionSystem_SymmetricDirected}), we even obtain $e_{vu} \in E$ \textit{and} $e_{uv} \in E$.
    Therefore, by
    (1) the conservation of energy (Definition \ref{definition_HydraulicPrinciples_ConservationOfEnergy}),
    (2) basic transformations and
    (3) the fact that the geometric pipe features are symmetric in the sense that $r_{vu} = r_{uv} > 0$ holds for all $v \in V$ and $u \in \Nbh(v)$,
    we
    obtain
    \begin{align}
    \label{align_intheorem_RedundancyOfConservationOfFlows}
    \begin{split}
        r_{vu} \sgn(q_{vu}) |q_{vu}|^x 
        \overset{\text{(1)}}{=}&~
        h_v - h_u
        \\
        \overset{\text{(2)}}{=}&~
        - (h_u - h_v)
        \\
        \overset{\text{(1)}}{=}&~
        - r_{uv} \sgn(q_{uv}) |q_{uv}|^x
        \\
        \overset{\text{(3)}}{=}&~
        - r_{vu} \sgn(q_{uv}) |q_{uv}|^x.
    \end{split}
    \end{align}

    \noindent By 
    dividing both sides of Equation \eqref{align_intheorem_RedundancyOfConservationOfFlows} by $r_{vu} > 0$, 
    taking the absolute value and 
    using the fact that the function $\R_{\geq 0} \rightarrow \R_{\geq 0}, t \mapsto t^x$ is injective, 
    we obtain 
    $|q_{vu}| = |q_{uv}|$.

    If  $q_{vu} = 0$ holds, we can conclude that $q_{uv} = 0$ holds by the positive definiteness of the absolute value. 
    Consequently, we obtain $q_{vu} = 0 = - 0 = - q_{uv}$.

    If $q_{vu} \neq 0$ holds, we can divide both sides of Equation \eqref{align_intheorem_RedundancyOfConservationOfFlows} by $r_{vu} > 0$ and $|q_{vu}|^x > 0$ and obtain
    $\sgn(q_{vu}) = - \sgn(q_{uv})$.
    Bringing it all together,
    \begin{align*}
        \sgn(q_{vu}) ~ q_{vu}
        =
        |q_{vu}|
        =
        |q_{uv}|
        = 
        \sgn(q_{uv}) ~ q_{uv}
        =
        - \sgn(q_{vu}) ~ q_{uv}
    \end{align*}

    \noindent holds.
    Similarly, by dividing both sides of the equation by $\sgn(q_{vu}) \neq 0$, 
    we obtain $q_{vu} = -q_{uv}$.
\end{proof}

\subsection{Proof of Theorem \ref{theorem_HydraulicPrinciples_MatrixFormulation}}

\begin{theoremNoNb}[Theorem  \ref{theorem_HydraulicPrinciples_MatrixFormulation}]
    The hydraulic states $(\hm,\qm,\dm)$ of a \gls{WDS} $\G = (V,E)$ are physically correct \gls{iff}
    \begin{align}
    \label{align_HydraulicPrinciples_ConservationOfEnergy_MatrixForm_appendix}
        \Bm^T \hm &= \Dm \qm, \\
    \label{align_HydraulicPrinciples_ConservationOfMass_MatrixForm_appendix}
        \Bm_{\Vc \cdot} \qm &= - 2\dm
    \end{align}

    \noindent holds, where $\Bm$ is the incidence matrix
    \begin{align*}
        \Bm 
        :=
        \left(
        \1_{\{ \exists u \in V: ~ e = e_{vu} \}} - \1_{\{ \exists u \in V: ~ e = e_{uv} \}}
        \right)_{ \substack{v \in V \\ e \in E} }
        \in 
        \R^{\n{n} \times \n{e}}
        \text{ of } \G,
    \end{align*}

    \noindent $\Bm_{\Vc \cdot}$ is the incidence matrix $\Bm$ limited to only the rows corresponding to the consumer nodes $\Vc$
    and $\Dm = \Dm(\qm)$ is the flow-dependent diagonal matrix of nonlinearities
    \begin{align*}
        \Dm
        =
        \left(
        \delta_{e_1 e_2} \cdot r_{e_1} |q_{e_1}|^{x-1}
        \right)_{ \substack{e_1 \in E \\ e_2 \in E} }
    \end{align*}

    \noindent 
    with $\delta$ the Kronecker delta, $x$ the Hazen-Williams constant and $r_e$ and $q_e$ the resistance coefficient of and the flow along the edge $e \in E$, respectively (cf. Definition \ref{align_HydraulicPrinciples_ConservationOfEnergy}).
\end{theoremNoNb}

\begin{proof}
    By Lemma \ref{lemma_RedundancyOfHydraulicPrinciples_ConservationOfFlows}, 
    the hydraulic states $(\hm,\qm,\dm)$ are physically correct \gls{iff} they satisfy the properties \eqref{align_HydraulicPrinciples_ConservationOfEnergy} and \eqref{align_HydraulicPrinciples_ConservationOfMass}, i.e.,
    \gls{iff}
    \begin{align*}
        h_v - h_u 
        &= 
        r_{vu} \sgn(q_{vu}) |q_{vu}|^x 
        = 
        r_{vu} q_{vu} |q_{vu}|^{x-1}
        & \quad & \text{for all } e = e_{vu} \in E \text{ and}\\
        \sum_{u \in \mathcal{N}(v)} q_{vu} 
        &= 
        - d_v 
        & \quad & \text{for all } v \in \Vc
    \end{align*}

    \noindent hold. Especially to mention, Lemma \ref{lemma_RedundancyOfHydraulicPrinciples_ConservationOfFlows} states that a triple $(\hm,\qm,\dm)$ that satisfies the properties \eqref{align_HydraulicPrinciples_ConservationOfEnergy} and \eqref{align_HydraulicPrinciples_ConservationOfMass} already satisfies the conservation of flows, that is, 
    \begin{align*}
        (\hm,\qm,\dm)
        \in
        U 
        := 
        \{ (\hm,\qm,\dm) \in \R^{\n{c} \times \n{e}} ~|~ q_{uv} = -q_{vu} ~\forall v \in V \text{ and } u \in \Nbh(v)\}
    \end{align*}

    \noindent holds with $U$ the subspace on which the conservation of flows (Definition \ref{definition_HydraulicPrinciples_ConservationOfFlows}) holds.
    Consequently, it suffices to show that Equation \eqref{align_HydraulicPrinciples_ConservationOfEnergy_MatrixForm_appendix} and \eqref{align_HydraulicPrinciples_ConservationOfMass_MatrixForm_appendix} are the matrix-formulations of Equation \eqref{align_HydraulicPrinciples_ConservationOfEnergy} and \eqref{align_HydraulicPrinciples_ConservationOfMass}, respectively, given that $(\hm,\qm,\dm) \in U$ holds.
    
    By 
    (1) definition of the matrices and vectors (cf. Definition \ref{definition_HydraulicStates}),
    (2) matrix multiplication,
    (3) basic transformations and 
    (4) eventually renaming the variables,
    \begin{align*}
        \Bm^T \hm
        \overset{\text{(1)}}{=}&~
        \left(
        \1_{\{ \exists u \in V: ~ e = e_{vu} \}} - \1_{\{ \exists u \in V: ~ e = e_{uv} \}}
        \right)_{ \substack{e \in E \\ v \in V} }
        ~
        (h_v)_{v \in V}
        \\
        \overset{\text{(2)}}{=}&~
        \left(
        \sum_{v \in V}
        h_v
        \cdot
        \left(
        \1_{\{ \exists u \in V: ~ e = e_{vu} \}} 
        -
        \1_{\{ \exists u \in V: ~ e = e_{uv} \}}
        \right)
        \right)_{e \in E}
        \\
        \overset{\text{(3,4)}}{=}&~
        \left(
        \sum_{v \in V}
        h_v
        \cdot 
        \1_{\{ \exists u \in V: ~ e = e_{vu} \}} 
        -
        \sum_{u \in V}
        h_v
        \cdot
        \1_{\{ \exists v \in V: ~ e = e_{vu} \}}
        \right)_{e = e_{vu} \in E}
        \\
        \overset{\text{(3)}}{=}&~
        \left(
        h_v
        -
        h_u
        \right)_{e_{vu} \in E}
    \end{align*}

    \noindent and
    \begin{align*}
        \Dm \qm
        \overset{\text{(1)}}{=}&~
        \left(
        \delta_{e_1 e_2} \cdot r_{e_1} |q_{e_1}|^{x-1}
        \right)_{ \substack{e_1 \in E \\ e_2 \in E} }
        ~
        (q_{e_2})_{e_2 \in E}
        \\
        \overset{\text{(2)}}{=}&~
        \left(
        \sum_{e_2 \in E}
        \delta_{e_1 e_2} \cdot r_{e_1} q_{e_2} |q_{e_1}|^{x-1}
        \right)_{e_1 \in E}
        \\
        \overset{\text{(3)}}{=}&~
        \left(
        r_{e_1} q_{e_1} |q_{e_1}|^{x-1}
        \right)_{e_1 \in E}
        \\
        \overset{\text{(4)}}{=}&~
        \left(
        r_{e_{vu}} q_{vu} |q_{vu}|^{x-1}
        \right)_{e_{vu} \in E}
    \end{align*}

    \noindent represent the first property \eqref{align_HydraulicPrinciples_ConservationOfEnergy} (conservation of energy, Definition \ref{definition_HydraulicPrinciples_ConservationOfEnergy}).
    Moreover, by introducing the modified incidence matrix (where we drop all values equal to -1)\footnote{
        For clarity, the matrix $\Bmtilde_{\Vc \cdot}$ is the matrix $\Bmtilde$ limited to only those rows corresponding to consumer nodes $v \in \Vc$.
    }
    \begin{align*}
        \Bmtilde 
        :=
        \left(
        \1_{\{ \exists u \in V: ~ e = e_{vu} \}} 
        \right)_{ \substack{v \in V \\ e \in E} }
        \in 
        \R^{\n{n} \times \n{e}}
        \text{ of } \G
    \end{align*}
    
    \noindent and by additionally using
    (5) the definition of neighborhoods (cf. Remark \ref{definition_Neighborhoods}) and
    (6) the fact that the graph $\G$ is symmetric directed (cf. Definition \ref{definition_WaterDistributionSystem_SymmetricDirected}),
    \begin{align*}
        \Bmtilde_{\Vc \cdot} \qm
        \overset{\text{(1)}}{=}&~
        \left(
        \1_{\{ \exists u \in V: ~ e = e_{vu} \}} 
        \right)_{ \substack{v \in \Vc \\ e \in E} }
        ~
        (q_e)_{e \in E}
        \\
        \overset{\text{(2)}}{=}&~
        \left(
        \sum_{e \in E}
        q_e 
        \cdot
        \1_{\{ \exists u \in V: ~ e = e_{vu} \}} 
        \right)_{v \in \Vc}
        \\
        \overset{\text{(5)}}{=}&~
        \left(
        \sum_{u \in \Nbhout(v)}
        q_{vu} 
        \right)_{v \in \Vc}
        \\
        \overset{\text{(6)}}{=}&~
        \left(
        \sum_{u \in \Nbh(v)}
        q_{vu} 
        \right)_{v \in \Vc}
    \end{align*}
    
    \noindent and $-\dm$ represent the second property \eqref{align_HydraulicPrinciples_ConservationOfMass} (conservation of mass, Definition \ref{definition_HydraulicPrinciples_ConservationOfMass}).

    Finally, in order to get rid of the rather unknown modified incidence matrix $\Bmtilde$, we show that $2 \Bm_{\Vc \cdot} \qm = \Bm_{\Vc \cdot} \qm$ holds for all $(\hm,\qm,\dm) \in U$:
    Along similar lines, 
    by additionally using
    (7) the choice of $(\hm,\qm,\dm) \in U$,
    we obtain

    \begin{align*}
        \Bm_{\Vc \cdot} \qm
        \overset{\text{(1)}}{=}&~
        \left(
        \1_{\{ \exists u \in V: ~ e = e_{vu} \}} - \1_{\{ \exists u \in V: ~ e = e_{uv} \}}
        \right)_{ \substack{v \in \Vc \\ e \in E} }
        ~
        (q_e)_{e \in E}
        \\
        \overset{\text{(2)}}{=}&~
        \left(
        \sum_{e \in E}
        q_e 
        \cdot
        \left(
        \1_{\{ \exists u \in V: ~ e = e_{vu} \}} 
        -
        \1_{\{ \exists u \in V: ~ e = e_{uv} \}}
        \right)
        \right)_{v \in \Vc}
        \\
        \overset{\text{(3)}}{=}&~
        \left(
        \sum_{e \in E}
        q_e
        \cdot 
        \1_{\{ \exists u \in V: ~ e = e_{vu} \}} 
        -
        \sum_{e \in E}
        q_e
        \cdot
        \1_{\{ \exists u \in V: ~ e = e_{uv} \}}
        \right)_{v \in \Vc}
        \\
        \overset{\text{(5)}}{=}&~
        \left(
        \sum_{u \in \Nbhout(v)}
        q_{vu} 
        -
        \sum_{u \in \Nbhin(v)}
        q_{uv} 
        \right)_{v \in \Vc}
        \\
        \overset{\text{(6)}}{=}&~
        \left(
        \sum_{u \in \Nbh(v)}
        q_{vu} 
        -
        \sum_{u \in \Nbh(v)}
        q_{uv} 
        \right)_{v \in \Vc}
        \\
        \overset{\text{(7)}}{=}&~
        \left(
        \sum_{u \in \Nbh(v)}
        q_{vu} 
        +
        \sum_{u \in \Nbh(v)}
        q_{vu} 
        \right)_{v \in \Vc}
        \\
        \overset{\text{(3)}}{=}&~
        \left(
        2
        \sum_{u \in \Nbh(v)}
        q_{vu} 
        \right)_{v \in \Vc}
        \\
        =&~
        2 \Bmtilde_{\Vc \cdot} \qm,
    \end{align*}

    \noindent which proves the equality.
    Since in either direction, by Lemma \ref{lemma_RedundancyOfHydraulicPrinciples_ConservationOfFlows}, $(\hm,\qm,\dm) \in U$ holds, this is all what was left to show.
\end{proof}

\subsection{Proof of Theorem \ref{theorem_RankOfSubmatricesOfTheIncidenceMatrix}}

\begin{theoremNoNb}[Theorem \ref{theorem_RankOfSubmatricesOfTheIncidenceMatrix}]
    Let $\G = (V,E)$ be a finite, connected graph.
    Let $S \subsetneq V$ be a non-empty subset of nodes. 
    Let $\Bm \in \R^{\n{n} \times \n{e}}$ be the incidence matrix (as defined in Theorem \ref{theorem_HydraulicPrinciples_MatrixFormulation}).
    Then $\rk(\Bm_{S \cdot}) = |S|$ holds.
\end{theoremNoNb}

\begin{proof}
    A well-known property of the incidence matrix is that the entries of each column $\Bm_{\cdot e}$ for each corresponding edge $e \in E$ add up to zero. 
    Based on that, a common result in graph theory is that taking the sum $\sum_{v \in V} \Bm_{v \cdot}$ over the row vectors $\Bm_{v \cdot}$ for corresponding nodes $v \in V$ yields the zero vector and that $\rk(\Bm) = \n{n} - 1$ holds (cf. \cite{Bapat2014Graphtheory}).
    We make the following observations to prove an even stronger statement:

    \begin{itemize}
        \item \textit{Observation 1:}
        By definition of the incidence matrix, each column $\Bm_{\cdot e}$ corresponding to an edge $e \in E$ holds exactly one 1, one -1 and the remaining entries are 0. 
        More precisely, for each $e \in E$ there exists exactly one pair of (neighboring) nodes $v,u \in V$, $v \neq u$, such that $\Bm_{ve} = b_{ve} = 1 \neq 0$, $\Bm_{ue} = b_{ue} = - b_{ve} = -1 \neq 0$ and $\Bm_{we} = b_{we} = 0$ holds for all $w \in V \setminus \{v,u\}$.\footnote{
            If $e = e_{vu} \in E$, the pair we are looking for is $v,u \in V$. 
        }

        \item \textit{Observation 2:}
        By definition of the incidence matrix and due to the fact that the graph $\G$ is connected, each row $\Bm_{v \cdot}$ corresponding to a node $v \in V$ holds at least one entry not being equal to zero. 
        More precisely, for each $v \in V$ there exists an edge $e \in E$, such that $\Bm_{ve} = b_{ve} \neq 0$ holds.
    \end{itemize}
    
    \noindent We will first show that if we remove a row $\Bm_{v \cdot}$ corresponding to a node $v \in V$ from the incidence matrix $\Bm$, the rank of the remaining incidence matrix $\Bm_{V \setminus \{v\} \cdot}$ remains unchanged, or, in other words, that the $\n{n} - 1$ remaining rows remain linearly independent.

    Therefore, we assume that
    \begin{align}
    \label{align_intheorem_RankOfSubmatricesOfTheIncidenceMatrix_LinearCombination}
        \sum_{u \in V \setminus \{v\}} \lambda_u \Bm_{u \cdot}
        =
        \sum_{u \in V \setminus \{v\}} \lambda_u \left( b_{ue} \right)_{e \in E}
        =
        \zerom{\n{e}}
    \end{align}

    \noindent holds for $\lambda_u \in \R$ for all $u \in V \setminus \{v\}$.
    We now show by induction to $k \in \N$, where $k$ indicates the $k$-hop neighborhood $\Nbh_k(v) = \{ u \in V ~|~ d(v,u) = k\}$ of $v \in V$, that $\lambda_u = 0$ needs hold for all $u \in V \setminus \{v\}$:
    \\
    \\\textit{Induction base:}
    Let $k = 1$ and $u \in \Nbh_1(v) = \Nbh(v) = \{ u \in V ~|~ \exists e \in E: ~ e = e_{vu} \text{ or } e = e_{uv} \}$ (cf. Remark \ref{definition_Neighborhoods}) hold.
    Due to \textit{observation 1}, there exists an edge $e \in E$ (namely, the edge $e = e_{vu}$ or $e = e_{uv}$ that connects $v$ and $u$) such that $b_{ue} = - b_{ve} \neq 0$ and $b_{we} = 0$ holds for all $w \in V \setminus \{v,u\}$.
    Consequently, the summand $\lambda_u b_{ue}$ is the one of all summands $\{ \lambda_w b_{we} ~|~ w \in V \setminus \{v\}\}$ that is potentially is non-zero.
    Therefore, in order to 
    \begin{align*}
        \sum_{w \in V \setminus \{v\}} \lambda_w b_{we}
        =
        \lambda_u 
        \underbrace{b_{ue}}_{\neq 0}
        =
        0
    \end{align*}

    \noindent to be satisfied, $\lambda_u = 0$ needs to hold.
    \\
    \\\textit{Induction hypothesis:}
    Let $k \in \N$ and $u \in \Nbh_k(v)$ hold. 
    Then $\lambda_u = 0$ holds.
    \\
    \\\textit{Induction step:}
    Let $u \in \Nbh_{k+1}(v)$ hold. 
    Since the graph $\G$ is connected, there exists a $\tilde{u} \in \Nbh(u) = \Nbh_1(u)$ such that $\tilde{u} \in \Nbh_k(v)$ holds.
    An analogous argument as in the \textit{Induction base} shows that
    due to \textit{observation 1}, there exists an edge $e \in E$ (namely, the edge $e = e_{\tilde{u} u}$ or $e = e_{u \tilde{u}}$ that connects $\tilde{u}$ and $u$) such that $b_{ue} = - b_{\tilde{u} e} \neq 0$ and $b_{we} = 0$ holds for all $w \in V \setminus \{\tilde{u},u\}$.
    Additionally, according to the \textit{Induction hypotheses}, $\lambda_{\tilde{u}} = 0$ and thus $\lambda_{\tilde{u}} b_{\tilde{u} e} = 0$ holds.
    Consequently, the summand $\lambda_u b_{ue}$ is the only one of all summands $\{ \lambda_w b_{we} ~|~ w \in V \setminus \{v\}\}$ that is potentially is non-zero.
    Therefore, we can again conclude that $\lambda_u = 0$ needs to hold.
    \\
    \\In conclusion, we can iteratively conclude that in order for Equation \eqref{align_intheorem_RankOfSubmatricesOfTheIncidenceMatrix_LinearCombination} to be satisfied, $\lambda_u = 0$ needs to hold for all $u \in V \setminus \{v\}$. 
    Therefore, 
    if we remove a row $\Bm_{v \cdot}$ corresponding to a node $v \in V$ from the incidence matrix $\Bm$, the $\n{n} - 1$ remaining rows remain linearly independent and by this, the rank of the remaining incidence matrix $\Bm_{V \setminus \{v\} \cdot}$ remains unchanged.  
    \\
    \\Finally, along similar lines, we can argue that removing multiple rows $\Bm_{v \cdot}$ repeatedly corresponding to multiple nodes $v \in S^C = V \setminus (S \setminus V)$ does not change that, i.e., leaves linearly independent rows.
    In fact, the only thing that can change in the argumentation above is that eventually, due to the removal of previous rows, there can exist columns $\Bm_{\cdot e}$ corresponding to edges $e \in E$ in the remaining incidence matrix $\Bm_{V \setminus S^C \cdot}$ that consists of zeros only.\footnote{
        More intuitively, these columns correspond to edges $e = e_{vu} \in E$ of which both nodes $v, u \in V$ have been removed already in previous steps.
    } 
    More precisely, there can exist edges $e \in E$ for which all summands $\{ \lambda_w b_{we} ~|~ w \in V \setminus S^C\}$ are zero already.
    However, by \textit{observation 2}, for each node $u \in V \setminus S^C$ and by this, for each summand $\lambda_u \Bm_{u \cdot}$, there exists an(other) edge $e \in E$, such that $\Bm_{ve} = b_{ve} \neq 0$ holds and which -- analogously to the lines above -- will enforce $\lambda_u = 0$ to hold.
    \\
    \\In summary, after removing multiple rows repeatedly according to multiple nodes $S^C = V \setminus (S \setminus V)$, the remaining $|S|$ rows remain linearly independent rows. 
    This causes $\Bm_{V \setminus S^C} = \Bm_{V \setminus (V \setminus S) \cdot} = \Bm_{S \cdot} \in \R^{|S| \times \n{e}}$ to have full rank $\rk(\Bm_{S \cdot}) = |S|$.
\end{proof}

\subsection{Proof of Corollary \ref{corollary_DecompositionOfTheInciddenceMatrix}}

\begin{corollaryNoNb}[Corollary \ref{corollary_DecompositionOfTheInciddenceMatrix}]
    Let $\G = (V,E)$ be a finite, connected graph.
    Let $S \subsetneq V$ be a non-empty subset of nodes. 
    Let $\Bm \in \R^{\n{n} \times \n{e}}$ be the incidence matrix (as defined in Theorem \ref{theorem_HydraulicPrinciples_MatrixFormulation}).
    Then there exist two subsets $\Ei \subset E$ and $\Ed \subset E$ of edges, such that
    
    \begin{enumerate}
        \item $|\Ei| = |S| =: \n{s}$, $|\Ed| = |E| - |S| = \n{e} - \n{s}$ and $E = \Ei \sqcup \Ed$ hold,
        
        \item $\rk(\Bm_{S \cdot}) = \rk(\Bm_{S \Ei}) = \n{s}$ holds,
        
        \item $\Bm_{S \Ei} \in \R^{\n{s} \times \n{s}}$ is invertible and

        \item $\rk(\Bm_{\cdot \Ei}) = \n{s}$ holds.
    \end{enumerate} 

    \noindent More precisely, 
    $\Ei$ is a subset of edges $E$ 
    whose corresponding columns $\Bm_{S \Ei}$ 
    in the limited incidence matrix $\Bm_{S \cdot}$ 
    are (a possible choice of) $\n{s} = |S|$ linearly \textbf{i}ndependent ones; 
    $\Ed = E \setminus \Ei$ is the subset of edges $E$
    whose corresponding columns $\Bm_{S \Ed}$ 
    in the incidence matrix $\Bm_{S \cdot}$ are the $\n{e} - \n{s} = |E| - |S|$ remaining linearly \textbf{d}ependent ones. 
    Even more, 
    $\Ei$ is also a subset of edges $E$
    whose corresponding columns $\Bm_{\cdot \Ei}$
    in the incidence matrix $\Bm$
    are $\n{s}$ of at most $\n{n} - 1$ linearly independent ones.
\end{corollaryNoNb}

\begin{proof}
    By Theorem \ref{theorem_RankOfSubmatricesOfTheIncidenceMatrix}, $\rk(\Bm_{S \cdot}) = |S| =: \n{s}$ holds, that is, the matrix $\Bm_{S \cdot} \in \R^{\n{s} \times \n{e}}$ has $\n{s}$ linearly independent columns.
    Therefore, we can divide the $\n{e}$ columns of the matrix $\Bm_{S \cdot} \in \R^{\n{s} \times \n{e}} = \R^{\n{s} \times (\n{s} + (\n{e}-\n{s}))}$ into $\n{s}$ linearly independent columns and $\n{e} - \n{s}$ remaining columns.\footnote{
        Note that for the matrix $\Bm_{S \cdot} \in \R^{\n{s} \times \n{e}}$, $\n{s} = \rk(\Bm_{S \cdot}) \leq \min\{\n{s},\n{e}\} \leq \n{s}$ holds, i.e., $\n{s} = \min\{\n{s},\n{e}\} \leq \n{e}$.
        Intuitively, this follows from the connectedness of $\G$ and the choice of $S \subsetneq V$:

        Since $\G$ is a connected graph $\n{n} - 1 \leq \n{e}$ holds.
        Moreover, by the choice of $S \subsetneq V$, $\n{s} < \n{n}$ or equivalently (since $\n{s}, \n{n} \in \N$), $\n{s} \leq \n{n} - 1$ holds.
        In summary, we obtain $\n{s} \leq \n{e}$.
    } 
    The latter $\n{e} - \n{s}$ columns depend linearly on the former $\n{s}$ columns. 
    We assign $\Ei$ to the subset of edges $E$ that correspond to a possible choice such $\n{s}$ linearly \textbf{i}ndependent columns and assign $\Ed$ to the subset of edges $E$ that correspond to the $\n{e} - \n{s}$ linearly \textbf{d}ependent columns
    (-- note that these choices are not unique).

    By construction, Corollary \ref{corollary_DecompositionOfTheInciddenceMatrix}.1 and \ref{corollary_DecompositionOfTheInciddenceMatrix}.2 hold. 
    Moreover, since $\Bm_{S \Ei} \in \R^{\n{s} \times \n{s}}$ is a quadratic matrix, by basic results from linear algebra, $\rk(\Bm_{S \Ei}) = \n{s}$ implies that $\Bm_{S \Ei}$ is invertible, yielding Corollary \ref{corollary_DecompositionOfTheInciddenceMatrix}.3.
    Finally, in order to prove Corollary \ref{corollary_DecompositionOfTheInciddenceMatrix}.4,
    let us assume that $\rk(\Bm_{\cdot \Ei}) \neq \n{s}$ holds. 
    In that case,
    the $\n{s}$ columns $\Bm_{\cdot e}$ 
    of the limited incidence matrix $\Bm_{\cdot \Ei} \in \R^{\n{n} \times \n{s}}$ 
    for all $e \in \Ei$
    are not linearly independent. 
    Consequently, there exist an $e \in \Ei$ and a $\lambda_e \neq 0$, such that
    \begin{align*}
        \sum_{e \in \Ei}
        \lambda_e \Bm_{\cdot e}
        =
        \zerom{\n{n}}
    \end{align*}

    \noindent holds.
    Out of these $\n{n}$ equations, let us consider the $\n{s}$ equations determined by the limited incidence matrix $\Bm_{S\Ei} \in \R^{\n{s} \times \n{s}}$ studied above:
    \begin{align*}
        \sum_{e \in \Ei}
        \lambda_e \Bm_{S e}
        =
        \zerom{\n{s}}.
    \end{align*}

    \noindent By the fact that 
    the columns $\Bm_{S e}$
    of the limited incidence matrix $\Bm_{S\Ei} \in \R^{\n{s} \times \n{s}}$ 
    for all $e \in \Ei$
    are linearly independent,
    $\lambda_e = 0$ for all $e \in \Ei$ must follow
    -- a contradiction.
\end{proof}

\subsection{Proof of Theorem \ref{theorem_ExistenceAndUniquenessOfHydraulicStates_ReservoirHeadsAndConsumerHeads}}

\begin{theoremNoNb}[Theorem \ref{theorem_ExistenceAndUniquenessOfHydraulicStates_ReservoirHeadsAndConsumerHeads}]
    Let $\G = (V,E)$ be a \gls{WDS} and 
    $\hm = \hm_V = [\hm_\Vr^T,\hm_\Vc^T]^T$ 
    known 
    (reservoir and consumer) heads.
    \\There exists exactly one 
    tuple $(\qm,\dm)$ 
    of flows and demands
    such that the triple 
    $(\hm,\qm,\dm)$ 
    of hydraulic states is physically correct \gls{wrt} these side constraints.
\end{theoremNoNb}

\begin{proof}
    Given 
    (reservoir and consumer) heads $\hm \in \R^\n{n}$, 
    Equation 
    \eqref{align_HydraulicPrinciples_ConservationOfEnergy_MatrixForm} and 
    \eqref{align_HydraulicPrinciples_ConservationOfMass_MatrixForm} 
    can be transformed into the \gls{SNLE}
    \begin{align}
    \label{align_SLE_FlowsAndDemandsUnkown}
        \underbrace{
        \begin{array}{c@{\hspace{-2pt}}@{\hspace{-2pt}}c}
            \raisebox{-2pt}{
            \text{\tiny $\overbrace{\phantom{............}}^{\n{e} \times \n{e}}$}
            } 
            & 
            \raisebox{-2pt}{
            \text{\tiny $\overbrace{\phantom{......................}}^{\n{e} \times \n{c}}$}
            }
            \\
            \multicolumn{2}{c}{
            \begin{bmatrix}
                \Dm             & \Zerom{\n{e}}{\n{c}} \\
                \Bm_{\Vc \cdot} & 2 \cdot \Idm{\n{c}}
            \end{bmatrix}
            } 
            \\
            \raisebox{7pt}{
            \text{\tiny $\underbrace{\phantom{............}}_{\n{c} \times \n{e}}$}
            }
            & 
            \raisebox{7pt}{
            \text{\tiny $\underbrace{\phantom{......................}}_{\n{c} \times \n{c}}$}
            }
        \end{array}
        }_{
        \in \R^{(\n{e}+\n{c}) \times (\n{e}+\n{c})}
        }
        ~\cdot
        \underbrace{
        \begin{array}{c}
            \raisebox{-2pt}{
            \text{\tiny $\overbrace{\phantom{.....}}^{\n{e} \times 1}$}
            }
            \\
            \begin{bmatrix}
                \qm \\
                \dm
            \end{bmatrix}
            \\
            \raisebox{7pt}{
            \text{\tiny $\underbrace{\phantom{.....}}_{\n{c} \times 1}$}
            }
        \end{array}
        }_{
        \in \R^{\n{e}+\n{c}}
        }
        =
        \underbrace{
        \begin{array}{c@{\hspace{0pt}}@{\hspace{0pt}}c}
            \raisebox{-2pt}{
            \text{\tiny \color{white} $\overbrace{\phantom{.....}}^{\n{e} \times 1}$}
            }
            & 
            \\
            \multirow{2}{*}{$
            \begin{bmatrix}
                \Bm^T \hm \\
                \zerom{\n{c}}
            \end{bmatrix}
            $}
            & 
            \text{$\}$ \tiny $\n{e} \times 1$}
            \\
            & 
            \text{$\}$ \tiny $\n{c} \times 1$}
            \\
            \raisebox{7pt}{
            \text{\tiny \color{white} $\underbrace{\phantom{.....}}^{\n{c} \times 1}$}
            }
            &
        \end{array}
        }_{
        \in \R^{\n{e}+\n{c}}
        }
        .
    \end{align}

    \noindent 
    \textbf{Existence:}
    \textit{Step 1:}
    \textit{There exists a $\qm \in \R^\n{e}$ such that $\Dm \qm = \Bm^T \hm$ holds.}

    Although the equation $\Dm \qm = \Bm^T \hm$ is non-linear, since $\Dm \in \R^{\n{e} \times \n{e}}$ is a diagonal matrix, it induces a one-to-one relation
    \begin{align*}
        h_v - h_u = r_{vu} \sgn(q_{vu}) |q_{vu}|^x
    \end{align*}
    
    \noindent between every flow $q_{vu} \in \R^\n{e}$ and two neighboring heads $h_v, h_u \in \R^\n{n}$ for all nodes $v \in V$ and $u \in \Nbh(v)$ (this is simply the component-wise version, i.e., Equation \eqref{align_HydraulicPrinciples_ConservationOfEnergy}, of the conservation of energy (Definition \ref{definition_HydraulicPrinciples_ConservationOfEnergy})).

    Since $r_{vu} > 0$ and $|q_{vu}|^x \geq 0$ holds, we can conclude that $\sgn(h_v-h_u) = \sgn(q_{vu})$ holds.
    If $0 = \sgn(h_v-h_u) = \sgn(q_{vu})$ holds, $q_{vu} = 0$ follows.
    If $0 \neq \sgn(h_v-h_u) = \sgn(q_{vu})$ holds, by
    the fact that $1 = \sgn(q_{vu}) \cdot \sgn(q_{vu}) = \sgn(q_{vu}) \cdot \sgn(h_v-h_u)$ holds, 
    we obtain
    \begin{align*}
        \sgn(q_{vu}) |q_{vu}|^x
        =&~
        r_{vu}^{-1} (h_v - h_u)
        \\
        =&~
        r_{vu}^{-1}
        \sgn(q_{vu}) \sgn(h_v-h_u) ~
        (h_v - h_u)
        \\
        =&~
        r_{vu}^{-1} 
        \sgn(q_{vu}) 
        |h_v - h_u|.
    \end{align*}

    \noindent Dividing both sides of the equation by $\sgn(q_{vu}) \neq 0$ and taking the power of $\frac{1}{x}$, we obtain
    $|q_{vu}| = \left(r_{vu}^{-1} |h_v - h_u|\right)^{\frac{1}{x}}$
    and finally, 
    \begin{align*}
        q_{vu}
        =
        \sgn(q_{vu})
        |q_{vu}|
        =
        \sgn(h_v - h_u)
        \left(
        r_{vu}^{-1} |h_v - h_u|
        \right)^{\frac{1}{x}}.
    \end{align*}
    
    \noindent 
    \textit{Step 2:}
    \textit{There exists a $\dm \in \R^\n{c}$ such that $2 \cdot \Idm{\n{c}} \dm = - \Bm_{\Vc \cdot} \qm$ holds.}

    Given that we have knowledge of the flows $\qm \in \R^\n{e}$ by \textit{Step 1}, 
    the equation $2 \cdot \Idm{\n{c}} \dm = - \Bm_{\Vc \cdot} \qm$ trivially translates to $\dm = -\frac{1}{2} \Bm_{\Vc \cdot} \qm$.
    \\
    \\In conclusion, 
    the tuple 
    $(\qm,\dm)$ 
    solves the 
    \gls{SNLE} \eqref{align_SLE_FlowsAndDemandsUnkown} 
    and therefore, 
    the triple 
    $(\hm,\qm,\dm)$
    is physically correct \gls{wrt} the given side constraints.
    \\
    \\\textbf{Uniqueness:}
    We have basically already seen in the existence part that 
    $\qm$ is uniquely determined by $\hm$ and that 
    $\dm$ is uniquely determined by $\qm$. 
    In a more detailed way, 
    let us assume that 
    there exist two
    tuples $(\qm_1,\dm_2) \neq (\qm_1,\dm_2)$
    of demands and flows 
    such that the triples
    $(\hm,\qm_1,\dm_1)$ and $(\hm,\qm_2,\dm_2)$
    of hydraulic states are physically correct.
    \\
    \\\textit{Step 1:} 
    If $\qm_1 \neq \qm_2$ holds, 
    there exists an $e = e_{vu} \in E$ for $v \in V, u \in \Nbh(v)$, 
    such that $q_{1vu} \neq q_{2vu}$ holds. 
    Since $(\hm,\qm_1,\dm_1)$ and $(\hm,\qm_2,\dm_2)$ are physically correct,
    \begin{align*}
        r_{vu} \sgn(q_{1vu}) |q_{1vu}|^x 
        =
        h_v - h_u 
        =
        r_{vu} \sgn(q_{2vu}) |q_{2vu}|^x 
    \end{align*}

    \noindent holds, 
    and by analogous arguments as given in the proof of Lemma \ref{lemma_RedundancyOfHydraulicPrinciples_ConservationOfFlows}, 
    we obtain 
    $q_{1vu} = q_{2vu}$ -- a contradiction. 
    Therefore, $\qm_1 = \qm_2$ must hold.
    \\
    \\\textit{Step 2:}
    If $\dm_1 \neq \dm_2$ holds,
    there exists a $v \in \Vc$, 
    such that $d_{1v} \neq d_{2v}$ holds. 
    Similarly, since $(\hm,\qm_1,\dm_1)$ and $(\hm,\qm_2,\dm_2)$ are physically correct and $\qm_1 = \qm_2$ holds by \textit{Step 1},
    \begin{align*}
        d_{1v} 
        = 
        -\sum_{u \in \mathcal{N}(v)} q_{1vu} 
        = 
        -\sum_{u \in \mathcal{N}(v)} q_{2vu} 
        = 
        d_{2v}
    \end{align*}

    \noindent holds -- a contradiction.
    Therefore,  $\dm_1 = \dm_2$ must hold.
\end{proof}

\subsection{Proof of Theorem \ref{theorem_ExistenceAndUniquenessOfHydraulicStates_ReservoirHeadsAndFlows}}

\begin{theoremNoNb}[Theorem \ref{theorem_ExistenceAndUniquenessOfHydraulicStates_ReservoirHeadsAndFlows}]
    Let $\G = (V,E)$ be a \gls{WDS} and 
    $\hm_\Vr$ and $\qm = \qm_E$ 
    known 
    reservoir heads and flows, respectively.
    \\
    If $\Dm \qm - \Bm_{\Vr \cdot}^T \hm_\Vr \in \im(\Bm_{\Vc \cdot}^T) \subset \R^{\n{e}}$ holds,
    there exists exactly one 
    tuple $(\hm_\Vc,\dm)$ 
    of consumer heads and demands 
    such that the triple 
    $(\hm,\qm,\dm) = ([\hm_\Vr^T,\hm_\Vc^T]^T,\qm,\dm)$ 
    of hydraulic states is physically correct \gls{wrt} these side constraints.
\end{theoremNoNb}

\begin{proof}
    Given 
    reservoir heads $\hm_\Vr \in \R^\n{r}$ and 
    flows $\qm \in \R^\n{e}$, 
    Equation 
    \eqref{align_HydraulicPrinciples_ConservationOfEnergy_MatrixForm_separatedHeads} and
    \eqref{align_HydraulicPrinciples_ConservationOfMass_MatrixForm_separatedHeads} 
    can be transformed into the \gls{SLE}
    \begin{align}
    \label{align_SLE_ConsumerHeadsAndDemandsUnkown}
        \underbrace{
        \begin{array}{c@{\hspace{-2pt}}@{\hspace{-2pt}}c}
            \raisebox{-2pt}{
            \text{\tiny $\overbrace{\phantom{................}}^{\n{e} \times \n{c}}$}
            }
            & 
            \raisebox{-2pt}{
            \text{\tiny $\overbrace{\phantom{......................}}^{\n{e} \times \n{c}}$}
            }
            \\
            \multicolumn{2}{c}{
            \begin{bmatrix}
                \Bm_{\Vc \cdot}^T    & \Zerom{\n{e}}{\n{c}} \\
                \Zerom{\n{c}}{\n{c}} & 2 \cdot \Idm{\n{c}}
            \end{bmatrix}
            } 
            \\
            \raisebox{7pt}{
            \text{\tiny $\underbrace{\phantom{................}}_{\n{c} \times \n{c}}$}
            }
            & 
            \raisebox{7pt}{
            \text{\tiny $\underbrace{\phantom{......................}}_{\n{c} \times \n{c}}$}
            }
        \end{array}
        }_{
        \in \R^{(\n{e}+\n{c}) \times 2 \n{c}}
        }
        \cdot
        \underbrace{
        \begin{array}{c}
            \raisebox{-2pt}{
            \text{\tiny $\overbrace{\phantom{..........}}^{\n{c} \times 1}$}
            }
            \\
            \begin{bmatrix}
                \hm_{\Vc} \\
                \dm
            \end{bmatrix}
            \\
            \raisebox{7pt}{
            \text{\tiny $\underbrace{\phantom{..........}}_{\n{c} \times 1}$}
            }
        \end{array}
        }_{
        \in \R^{2 \n{c}}
        }
        =
        \underbrace{
        \begin{array}{c@{\hspace{0pt}}@{\hspace{0pt}}c}
            \raisebox{-2pt}{
            \text{\tiny \color{white} $\overbrace{\phantom{.....}}^{\n{e} \times 1}$}
            }
            & 
            \\
            \multirow{2}{*}{$
            \begin{bmatrix}
                \Dm \qm - \Bm_{\Vr \cdot}^T \hm_\Vr \\
                - \Bm_{\Vc \cdot} \qm
            \end{bmatrix}
            $}
            & 
            \text{$\}$ \tiny $\n{e} \times 1$}
            \\
            & 
            \text{$\}$ \tiny $\n{c} \times 1$}
            \\
            \raisebox{7pt}{
            \text{\tiny \color{white} $\underbrace{\phantom{.....}}^{\n{c} \times 1}$}
            }
            &
        \end{array}
        }_{
        \in \R^{\n{e}+\n{c}}
        }
        .
    \end{align}

    \noindent 
    \textbf{Existence:}
    \textit{Step 1:}
    \textit{There exists a $\hm_\Vc \in \R^\n{c}$ such that $\Bm_{\Vc \cdot}^T \hm_\Vc = \Dm \qm - \Bm_{\Vr \cdot}^T \hm_\Vr$ holds.}

    Since $\Dm \qm - \Bm_{\Vr \cdot}^T \hm_\Vr \in \im(\Bm_{\Vc \cdot}^T) \subset \R^{\n{e}}$ holds, 
    by definition of the image $\im(\Bm_{\Vc \cdot}^T) = \{ \Bm_{\Vc \cdot}^T \hm_\Vc ~|~ \hm_\Vc \in \R^\n{c} \}$, there exists a $\hm_\Vc \in \R^\n{c}$ such that $\Bm_{\Vc \cdot}^T \hm_\Vc = \Dm \qm - \Bm_{\Vr \cdot}^T \hm_\Vr$ holds.
    \\
    \\
    \textit{Step 2:}
    \textit{There exists a $\dm \in \R^\n{c}$ such that $2 \cdot \Idm{\n{c}} \dm = - \Bm_{\Vc \cdot} \qm$ holds.}

    The equation $2 \cdot \Idm{\n{c}} \dm = - \Bm_{\Vc \cdot} \qm$ trivially translates to $\dm = -\frac{1}{2} \Bm_{\Vc \cdot} \qm$.
    \\
    \\In conclusion, 
    the tuple 
    $(\hm_\Vc,\dm)$ 
    solves the 
    \gls{SLE} \eqref{align_SLE_ConsumerHeadsAndDemandsUnkown} 
    and therefore, 
    the triple 
    $(\hm,\qm,\dm) = ([\hm_\Vr^T,\hm_\Vc^T]^T,\qm,\dm)$
    is physically correct \gls{wrt} the given side constraints.
    \\
    \\\textbf{Uniqueness:}
    By
    (1) basic properties from linear algebra and
    (2) Theorem \ref{theorem_RankOfSubmatricesOfTheIncidenceMatrix},
    we obtain
    \begin{align*}
        \rk \left(
        \begin{bmatrix}
            \Bm_{\Vc \cdot}^T    & \Zerom{\n{e}}{\n{c}} \\
            \Zerom{\n{c}}{\n{c}} & 2 \cdot \Idm{\n{c}}
        \end{bmatrix}
        \right)
        \overset{\text{(1)}}{=}
        \rk \left(
        \Bm_{\Vc \cdot}^T
        \right)
        +
        \rk \left(
        2 \cdot \Idm{\n{c}}
        \right)
        \overset{\text{(1)}}{=}
        \rk \left(
        \Bm_{\Vc \cdot}
        \right)
        +
        \rk \left(
        \Idm{\n{c}}
        \right)
        \overset{\text{(2)}}{=}
        2 \n{c}.
    \end{align*}

    \noindent Therefore, the solution found above is the only, unique solution.

        
\end{proof}

\subsection{Proof of Lemma \ref{lemma_ExistenceAndUniquenessOfHydraulicStates_OnReservoirHeadsAndFlowsCondition}}

\begin{lemmaNoNb}[Lemma \ref{lemma_ExistenceAndUniquenessOfHydraulicStates_OnReservoirHeadsAndFlowsCondition}]
    Let $\G = (V,E)$ be a \gls{WDS} and 
    $\hm_\Vr$ and $\qm = \qm_E$ 
    known 
    reservoir heads and flows, respectively.
    If there exist unobserved consumer heads $\hm_\Vc$ and demands $\dm$ such that
    the triple $(\hm,\qm,\dm) = ([\hm_\Vr^T,\hm_\Vc^T]^T,\qm,\dm)$ 
    of hydraulic states is physically correct \gls{wrt} these side constraints,
    then $\Dm \qm - \Bm_{\Vr \cdot}^T \hm_\Vr \in \im(\Bm_{\Vr \cdot}^T)$ holds.
\end{lemmaNoNb}

\begin{proof}
    Let us assume that 
    $\Dm \qm - \Bm_{\Vr \cdot}^T \hm_\Vr \not \in \im(\Bm_{\Vr \cdot}^T)$ 
    holds, that is, for all $\hm_\Vc \in \R^\n{c}$,
    $\Bm_{\Vr \cdot}^T \hm_\Vc \neq \Dm \qm - \Bm_{\Vr \cdot}^T \hm_\Vr$,
    or equivalently as seen in Equation \eqref{align_HydraulicPrinciples_ConservationOfEnergy_MatrixForm_separatedHeads},
    \begin{align*}
        \Bm^T \hm
        =
        \Bm_{\Vr \cdot}^T \hm_\Vc
        +
        \Bm_{\Vr \cdot}^T \hm_\Vr
        \neq
        \Dm \qm
    \end{align*}

    \noindent holds. However, since there exist unobserved consumer heads $\hm_\Vc$ and demands $\dm$, such that
    the triple $(\hm,\qm,\dm) = ([\hm_\Vr^T,\hm_\Vc^T]^T,\qm,\dm)$ 
    of hydraulic states is physically correct \gls{wrt} these side constraints, 
    according to Theorem \ref{theorem_HydraulicPrinciples_MatrixFormulation}, this is a contradiction to the conservation of energy.
\end{proof}

\subsection{Proof of Theorem \ref{theorem_ExistenceAndUniquenessOfHydraulicStates_ReservoirHeadsAndSomeFlows}}

\begin{theoremNoNb}[Theorem \ref{theorem_ExistenceAndUniquenessOfHydraulicStates_ReservoirHeadsAndSomeFlows}]
    Let $\G = (V,E)$ be a \gls{WDS} and 
    $\hm_\Vr$ and $\qm_\Ei$ 
    known 
    reservoir heads and flows, respectively,
    where -- as introduced above -- $\Ei$ is any subset of the edges $E$ corresponding to $\n{c}$ linearly independent columns of the limited incidence matrix $\Bm_{\Vc \cdot}$ of $\G$.
    \\There exists exactly one 
    triple $(\hm_\Vc,\qm_\Ed,\dm)$  
    of consumer heads, remaining flows and demands
    such that the triple 
    $(\hm,\qm,\dm) = ([\hm_\Vr^T,\hm_\Vc^T]^T,[\qm_\Ei^T,\qm_\Ed^T]^T,\dm)$ 
    of hydraulic states is physically correct \gls{wrt} these side constraints.
\end{theoremNoNb}

\begin{proof}
    Similar to as we have done with the heads earlier in Section \ref{subsection_ExistenceAndUniquenessOfHydraulicStates},
    instead of considering all flows $\qm \in \R^\n{e} = \R^{\n{c} + (\n{e} - \n{c})}$, 
    we can distinguish between the (given) flows $\qm_\Ei \in \R^\n{c}$ and the (remaining) flows $\qm_\Ed \in \R^{\n{e} - \n{c}}$. 
    Based on this distinction and after eventually resorting, we can rewrite the corresponding vectors and matrices as
    \begin{align*}
        \qm 
        =&~
        \begin{bmatrix}
            \qm_\Ei \\
            \qm_\Ed
        \end{bmatrix}
        \in \R^{\n{c}+(\n{e}-\n{c})},
        \\
        \Bm 
        =&~
        \begin{bmatrix}
            \Bm_{\cdot \Ei} & \Bm_{\cdot \Ed}
        \end{bmatrix}
        \in \R^{\n{n} \times (\n{c}+(\n{e}-\n{c}))} 
        = \R^{(\n{r}+\n{c}) \times (\n{c}+(\n{e}-\n{c}))},
        \\
        \Bm_{\Vr \cdot} 
        =&~
        \begin{bmatrix}
            \Bm_{\Vr \Ei} & \Bm_{\Vr \Ed}
        \end{bmatrix}
        \in \R^{\n{r} \times (\n{c}+(\n{e}-\n{c}))},
        \\
        \Bm_{\Vc \cdot} 
        =&~
        \begin{bmatrix}
            \Bm_{\Vc \Ei} & \Bm_{\Vc \Ed}
        \end{bmatrix}
        \in \R^{\n{c} \times (\n{c}+(\n{e}-\n{c}))},
        \\
        \Dm 
        =&~
        \begin{bmatrix}
            \Dm_{\Ei \Ei} & \Dm_{\Ei \Ed} \\
            \Dm_{\Ed \Ei} & \Dm_{\Ed \Ei}
        \end{bmatrix}
        =
        \begin{bmatrix}
            \Dm_{\Ei \Ei} & \Zerom{\n{c}}{\n{e}-\n{c}} \\
            \Zerom{\n{e}-\n{c}}{\n{c}} & \Dm_{\Ed \Ed}
        \end{bmatrix}
        \in \R^{(\n{c}+(\n{e}-\n{c})) \times (\n{c}+(\n{e}-\n{c}))},
    \end{align*}
    
    \noindent and therefore, the corresponding multiplications as
    \begin{align*}
        \Bm^T \hm
        =&
        \begin{bmatrix}
            \Bm_{\cdot \Ei}^T \\
            \Bm_{\cdot \Ed}^T
        \end{bmatrix}
        \cdot 
        \hm
        =
        \begin{bmatrix}
            \Bm_{\cdot \Ei}^T \hm \\
            \Bm_{\cdot \Ed}^T \hm
        \end{bmatrix}
        \in \R^{\n{c}+(\n{e}-\n{c})},
        \\
        \Bm_{\Vr \cdot}^T \hm_\Vr
        =&
        \begin{bmatrix}
            \Bm_{\Vr \Ei}^T \\
            \Bm_{\Vr \Ed}^T
        \end{bmatrix}
        \cdot 
        \hm_\Vr
        =
        \begin{bmatrix}
            \Bm_{\Vr \Ei}^T \hm_\Vr \\
            \Bm_{\Vr \Ed}^T \hm_\Vr
        \end{bmatrix}
        \in \R^{\n{c}+(\n{e}-\n{c})},
        \\
        \Bm_{\Vc \cdot}^T \hm_\Vc
        =&
        \begin{bmatrix}
            \Bm_{\Vc \Ei}^T \\
            \Bm_{\Vc \Ed}^T
        \end{bmatrix}
        \cdot 
        \hm_\Vc
        =
        \begin{bmatrix}
            \Bm_{\Vc \Ei}^T \hm_\Vc \\
            \Bm_{\Vc \Ed}^T \hm_\Vc
        \end{bmatrix}
        \in \R^{\n{c}+(\n{e}-\n{c})},
        \\
        \Dm \qm=&
        \begin{bmatrix}
            \Dm_{\Ei \Ei}              & \Zerom{\n{c}}{\n{e}-\n{c}} \\
            \Zerom{\n{e}-\n{c}}{\n{c}} & \Dm_{\Ed \Ed}
        \end{bmatrix}
        \cdot
        \begin{bmatrix}
            \qm_\Ei \\
            \qm_\Ed
        \end{bmatrix}
        =
        \begin{bmatrix}
            \Dm_{\Ei \Ei} \qm_\Ei \\
            \Dm_{\Ed \Ed} \qm_\Ed
        \end{bmatrix}
        \in \R^{(\n{c}+(\n{e}-\n{c}))} 
        \\
        \Bm \qm
        =&
        \begin{bmatrix}
            \Bm_{\cdot \Ei} & \Bm_{\cdot \Ed}
        \end{bmatrix}
        \cdot 
        \begin{bmatrix}
            \qm_\Ei \\
            \qm_\Ed
        \end{bmatrix}
        =
        \Bm_{\cdot \Ei} ~ \qm_\Ei + \Bm_{\cdot \Ed} ~ \qm_\Ed
        \in \R^\n{n},
        \\
        \Bm_{\Vc \cdot} \qm
        =&
        \begin{bmatrix}
            \Bm_{\Vc \Ei} & \Bm_{\Vc \Ed}
        \end{bmatrix}
        \cdot 
        \begin{bmatrix}
            \qm_\Ei \\
            \qm_\Ed
        \end{bmatrix}
        =
        \Bm_{\Vc \Ei} ~ \qm_\Ei + \Bm_{\Vc \Ed} ~ \qm_\Ed
        \in \R^\n{c},
    \end{align*}
    
    \noindent where according to according to Corollary \ref{corollary_DecompositionOfTheInciddenceMatrix}.3, $\Bm_{\Vc \Ei} \in \R^{\n{c} \times \n{c}}$, and therefore also $\Bm_{\Vc \Ei}^T \in \R^{\n{c} \times \n{c}}$, is an invertible matrix.
    By this, Equation \eqref{align_HydraulicPrinciples_ConservationOfEnergy_MatrixForm_separatedHeads} and \eqref{align_HydraulicPrinciples_ConservationOfMass_MatrixForm_separatedHeads} indeed translate to Equation \eqref{align_HydraulicPrinciples_ConservationOfEnergy_MatrixForm_separatedHeads_separatedFlows_independentFlows} to \eqref{align_HydraulicPrinciples_ConservationOfMass_MatrixForm_separatedHeads_separatedFlows} as defined above:
    \begin{align*}
        \Bm_{\cdot \Ei}^T \hm
        =&~
        \Bm_{\Vr \Ei}^T \hm_\Vr + \Bm_{\Vc \Ei}^T \hm_\Vc 
        ~~=
        \Dm_{\Ei \Ei} ~ \qm_\Ei, \\
        \Bm_{\cdot \Ed}^T \hm
        =&~
        \Bm_{\Vr \Ed}^T \hm_\Vr + \Bm_{\Vc \Ed}^T \hm_\Vc
        ~= 
        \Dm_{\Ed \Ed} \qm_\Ed, \\
        \Bm_{\Vc \cdot} \qm
        =&~
        \Bm_{\Vc \Ei} \qm_\Ei + \Bm_{\Vc \Ed} ~ \qm_\Ed 
        = 
        - 2\dm.
    \end{align*}

    \noindent 
    Given 
    reservoir heads $\hm_\Vr \in \R^\n{r}$ and 
    flows $\qm_\Ei \in \R^\n{c}$
    corresponding to linearly independent columns of the incidence matrix $\Bm$ of $\G$, 
    Equation 
    \eqref{align_HydraulicPrinciples_ConservationOfEnergy_MatrixForm_separatedHeads_separatedFlows_independentFlows} 
    to 
    \eqref{align_HydraulicPrinciples_ConservationOfMass_MatrixForm_separatedHeads_separatedFlows} 
    can be transformed into the \gls{SLE}
    \small
    \begin{align}
    \label{align_SLE_ConsumerHeadsAndDemandsAndDependendFlowsUnkown}
        \underbrace{
        \begin{array}{c@{\hspace{0pt}}@{\hspace{0pt}}c@{\hspace{0pt}}@{\hspace{0pt}}c}
            \raisebox{-2pt}{
            \text{\tiny $\overbrace{\phantom{................}}^{\# \n{c}}$}
            }
            & 
            \raisebox{-2pt}{
            \text{\tiny $\overbrace{\phantom{........................}}^{\# (\n{e} - \n{c})}$}
            }
            & 
            \raisebox{-2pt}{
            \text{\tiny $\overbrace{\phantom{........................}}^{\# \n{c}}$}
            }
            \\
            \multicolumn{3}{c}{
            \begin{bmatrix}
                \Bm_{\Vc \Ei}^T      & \Zerom{\n{c}}{(\n{e} - \n{c})}   & \Zerom{\n{c}}{\n{c}} \\
                \Bm_{\Vc \Ed}^T      & -\Dm_{\Ed \Ed}                   & \Zerom{(\n{e} - \n{c})}{\n{c}} \\
                \Zerom{\n{c}}{\n{c}} & \Bm_{\Vc \Ed}                    & 2 \cdot \Idm{\n{c}}
            \end{bmatrix}
            }
            \\
            \raisebox{7pt}{
            \text{\tiny \color{white} $\underbrace{\phantom{.....}}^{\phantom{.....}}$}
            }
            & 
            \raisebox{7pt}{
            \text{\tiny \color{white} $\underbrace{\phantom{.....}}^{\phantom{.....}}$}
            }
            & 
            \raisebox{7pt}{
            \text{\tiny \color{white} $\underbrace{\phantom{.....}}^{\phantom{.....}}$}
            }
        \end{array}
        }_{
        \substack{
        \in \R^{(\n{c}+(\n{e}-\n{c})+\n{c}) \times (\n{c}+(\n{e}-\n{c})+\n{c})} \\
        = ~ \R^{(\n{c}+\n{e}) \times (\n{c}+\n{e})}
        }
        }
        \hspace{-10pt}
        ~~~\cdot
        \hspace{-10pt}
        \underbrace{
        \begin{array}{c}
            \raisebox{-2pt}{
            \text{\tiny \color{white} $\overbrace{\phantom{.....}}^{\# \n{c}}$}
            }
            \\
            \begin{bmatrix}
                \hm_{\Vc} \\
                \qm_\Ed \\
                \dm
            \end{bmatrix}
            \\
            \raisebox{7pt}{
            \text{\tiny \color{white} $\underbrace{\phantom{.....}}^{\phantom{.....}}$}
            }
        \end{array}
        }_{
        \substack{
        \in \R^{\n{c}+(\n{e}-\n{c})+\n{c}} \\
        = ~ \R^{\n{c}+\n{e}}
        }
        }
        \hspace{-10pt}
        =
        \underbrace{
        \begin{array}{c@{\hspace{0pt}}@{\hspace{0pt}}l}
            \raisebox{-2pt}{
            \text{\tiny \color{white} $\overbrace{\phantom{.....}}^{\# \n{c}}$}
            }
            & 
            \\
            \multirow{2}{*}{$
            \begin{bmatrix}
                \Dm_{\Ei \Ei} \qm_\Ei - \Bm_{\Vr \Ei}^T \hm_\Vr \\
                - \Bm_{\Vr \Ei}^T \hm_\Vr \\
                - \Bm_{\Vc \Ei} \qm_\Ei
            \end{bmatrix}
            $}
            & 
            \text{$\}$ \tiny $\# \n{c}$}
            \\
            & 
            \text{$\}$ \tiny $\# (\n{e} - \n{c})$}
            \\
            & 
            \text{$\}$ \tiny $\# \n{c}$}
            \\
            \raisebox{7pt}{
            \text{\tiny \color{white} $\underbrace{\phantom{.....}}^{\phantom{.....}}$}
            }
            &
        \end{array}
        }_{
        \substack{
        \in \R^{\n{c}+(\n{e}-\n{c})+\n{c}} \\
        = ~ \R^{\n{c}+\n{e}}
        }
        }
        .
    \end{align}
    \normalsize

    \noindent 
    \textbf{Existence:}
    \textit{Step 1:}
    \textit{There exists a $\hm_\Vc \in \R^\n{c}$ such that $\Bm_{\Vc \Ei}^T \hm_\Vc = \Dm_{\Ei \Ei} \qm_\Ei - \Bm_{\Vr \Ei}^T \hm_\Vr$ holds.}

    Since by Corollary \ref{corollary_DecompositionOfTheInciddenceMatrix}, $\Bm_{\Vc \Ei} \in \R^{\n{c} \times \n{c}}$, and therefore also $\Bm_{\Vc \Ei}^T \in \R^{\n{c} \times \n{c}}$ is invertible, 
    the equation $\Bm_{\Vc \Ei}^T \hm_\Vc = \Dm_{\Ei \Ei} \qm_\Ei - \Bm_{\Vr \Ei}^T \hm_\Vr$ translates to
    \begin{align*}
        \hm_\Vc 
        = 
        \left( 
        \Bm_{\Vc \Ei}^T
        \right)^{-1}
        \left(
        \Dm_{\Ei \Ei} \qm_\Ei - \Bm_{\Vr \Ei}^T \hm_\Vr
        \right).
    \end{align*}
    
    \noindent
    \textit{Step 2:}
    \textit{There exists a $\qm_\Ed \in \R^{\n{e}-\n{c}}$ such that $\Dm_{\Ed \Ed} \qm_\Ed = \Bm_{\Vr \Ed}^T \hm_\Vr + \Bm_{\Vc \Ed}^T \hm_\Vc = \Bm_{\cdot \Ed}^T \hm$ holds.}

    Given that we have knowledge of all heads $\hm = [\hm_\Vr^T,\hm_\Vc^T]^T \in \R^\n{n}$ by \textit{Step 1},
    although the equation $\Dm_{\Ed \Ed} \qm_\Ed = \Bm_{\cdot \Ed}^T \hm$ is non-linear, since $\Dm_{\Ed \Ed} \in \R^{(\n{e}-\n{c}) \times (\n{e}-\n{c})}$ is a diagonal matrix, it induces a one-to-one relation
    \begin{align*}
        h_v - h_u = r_{vu} \sgn(q_{vu}) |q_{vu}|^x
    \end{align*}
    
    \noindent between every \textit{remaining} flow $q_{vu} \in \R^{(\n{e}-\n{c})}$ and two neighboring heads $h_v, h_u \in \R^\n{n}$ for corresponding nodes $v \in V$ and $u \in \Nbh(v)$ (this is simply the component-wise version, i.e., Equation \eqref{align_HydraulicPrinciples_ConservationOfEnergy}, of the conservation of energy (Definition \ref{definition_HydraulicPrinciples_ConservationOfEnergy})).

    We can proceed analogously to \textit{Step 1} in the proof of Theorem \ref{theorem_ExistenceAndUniquenessOfHydraulicStates_ReservoirHeadsAndConsumerHeads} to compute these remaining flows. 
    \\
    \\
    \textit{Step 3:}
    \textit{There exists a $\dm \in \R^\n{c}$ such that $2 \cdot \Idm{\n{c}} \dm = - \Bm_{\Vc \Ei} \qm_\Ei - \Bm_{\Vc \Ed} \qm_\Ed = - \Bm_{\Vc \cdot} \qm$ holds.}

    Given that we have knowledge of all flows $\qm = [\qm_\Ei^T,\qm_\Ed^T]^T \in \R^\n{e}$ by \textit{Step 2},
    the equation $2 \cdot \Idm{\n{c}} \dm = - \Bm_{\Vc \cdot} \qm$ trivially translates to $\dm = -\frac{1}{2} \Bm_{\Vc \cdot} \qm$.
    \\
    \\In conclusion, 
    the triple 
    $(\hm_\Vc,\qm_\Ed,\dm)$
    solves the 
    \gls{SLE} \eqref{align_SLE_ConsumerHeadsAndDemandsAndDependendFlowsUnkown} 
    and therefore, 
    the triple 
    $(\hm,\qm,\dm) = ([\hm_\Vr^T,\hm_\Vc^T]^T,[\qm_\Ei^T,\qm_\Ed^T]^T,\dm)$ 
    is physically correct \gls{wrt} the given side constraints.
    \\
    \\\textbf{Uniqueness:}
    We have basically already seen in the existence part that 
    $\hm_\Vc$ is uniquely determined by $\qm_\Ei$, 
    $\qm_\Ed$ is uniquely determined by $\hm = [\hm_\Vr^T,\hm_\Vc^T]^T$ and that 
    $\dm$ is uniquely determined by $\qm = [\qm_\Ei^T,\qm_\Ed^T]^T)$.
    In a more detailed way, $\hm_\Vc$ is uniquely determined by the invertibility of the matrix $\Bm_{\Vc \Ei}$.
    For the uniqueness of $\qm_\Ed$ and $\dm$, we can proceed analogously to the uniqueness arguments in the proof of Theorem \ref{theorem_ExistenceAndUniquenessOfHydraulicStates_ReservoirHeadsAndConsumerHeads}.
\end{proof}

\subsection{Proof of Lemma \ref{lemma_TheNonLinearityDefinesAStrictilyMonotoneOperator}}

\begin{lemmaNoNb}
    The function
    \begin{align*}
        f: \R^\n{e} \longrightarrow \R^\n{e}, ~
        \qm 
        \longmapsto 
        f(\qm) 
        := 
        \Dm \qm
        =
        \left(
        r_e q_e |q_e|^{x-1}
        \right)_{e \in E}
        =:
        \left(
        f_e(\qm)
        \right)_{e \in E}
    \end{align*}

    \noindent satisfies
    \begin{align*}
        \langle 
        f(\qm_1) - f(\qm_2), \qm_1 - \qm_2
        \rangle
        =
        \sum_{e \in E}
        x ~ r_e ~ (q_{1e} - q_{2e})^2
        \int_{0}^{1}
        |q_{2e} + t (q_{1e} - q_{2e})|^{x-1}
        ~dt
        \geq 
        0
    \end{align*}

    \noindent for all $\qm_1, \qm_2 \in \R^\n{e}$. Even more, it defines a strictly monotone operator, that is, it satisfies
    \begin{align*}
        \langle 
        f(\qm_1) - f(\qm_2), \qm_1 - \qm_2
        \rangle
        >
        0
    \end{align*}

    \noindent for all $\qm_1, \qm_2 \in \R^\n{e}$ with $\qm_1 \neq \qm_2$.
\end{lemmaNoNb}

\begin{proof}
    For all $\qm_1, \qm_2 \in \R^{\n{e}}$, by
    (1) definition of the scalar product on $\R^\n{e}$,
    (2) the fact that $f_e(\qm) = f_e(q_e)$ holds for all $e \in E$, 
    (3) the fundamental theorem of calculus applied to the functions 
    $[0,1] \rightarrow \R \rightarrow \R, 
    t \mapsto q_{2e} + t (q_{1e} - q_{2e}) \mapsto f_e(q_{2e} + t (q_{1e} - q_{2e}))$
    for all $e \in E$,
    (4) chain rule,
    (5) the fact that $\frac{d}{d q_e} f(q_e) = x r_e |q_e|^{x-1}$ holds,
    (6) basic transformations and
    (7) the choices of $x = 1.852 > 0$ and $r_e > 0$ for all $e \in E$ (cf. Definition \ref{definition_HydraulicPrinciples_ConservationOfEnergy}),
    we obtain
    \begin{align*}
        \langle 
        f(\qm_1) - f(\qm_2), \qm_1 - \qm_2
        \rangle
        \overset{\text{(1)}}{=}&~
        \sum_{e \in E}
        (q_{1e} - q_{2e}) \cdot (f_e(\qm_1) - f_e(\qm_2))
        \\
        \overset{\text{(2)}}{=}&~
        \sum_{e \in E}
        (q_{1e} - q_{2e}) \cdot (f_e(q_{1e}) - f_e(q_{2e}))
        \\
        \overset{\text{(3)}}{=}&~
        \sum_{e \in E}
        (q_{1e} - q_{2e}) 
        \int_{0}^{1}
        \frac{d}{d t}
        f_e(q_{2e} + t (q_{1e} - q_{2e}))
        ~dt
        \\
        \overset{\text{(4)}}{=}&~
        \sum_{e \in E}
        (q_{1e} - q_{2e}) 
        \int_{0}^{1}
        (q_{1e} - q_{2e})
        ~
        \frac{d}{d q_e}
        f_e(q_{2e} + t (q_{1e} - q_{2e}))
        ~dt
        \\
        \overset{\text{(5)}}{=}&~
        \sum_{e \in E}
        (q_{1e} - q_{2e})
        \int_{0}^{1}
        (q_{1e} - q_{2e})
        ~
        x ~ r_e ~ |q_{2e} + t (q_{1e} - q_{2e})|^{x-1}
        ~dt
        \\
        \overset{\text{(6)}}{=}&~
        \sum_{e \in E}
        \underbrace{
        \underbrace{
        x ~ r_e ~ (q_{1e} - q_{2e})^2
        }_{
        \geq 0
        }
        \int_{0}^{1}
        \underbrace{
        |q_{2e} + t (q_{1e} - q_{2e})|^{x-1}
        }_{
        \geq 0
        }
        ~dt
        }_{
        \geq 0
        }
        \geq
        0.
    \end{align*}    

    \noindent Even more, if additionally, $\qm_1 \neq \qm_2$ holds, there exists an $e \in E$ such that $q_{1e} \neq q_{2e}$ holds. In this case, the corresponding summand
    \begin{align*}
        \underbrace{
        \underbrace{
        x ~ r_e ~ (q_{1e} - q_{2e})^2
        }_{
        > 0
        }
        \int_{0}^{1}
        \underbrace{
        |q_{2e} + t (q_{1e} - q_{2e})|^{x-1}
        }_{
        > 0 
        \text{ for almost all } t \in [0,1]
        }
        ~dt
        }_{
        > 0 
        }
        \overset{\text{(6,7)}}{>}~
        0
    \end{align*}

    \noindent is strictly positive, and since the remaining summands are non-negative, the overall sum ist strictly positive, too.
\end{proof}

\subsection{Proof of Theorem \ref{theorem_ExistenceAndUniquenessOfHydraulicStates_ReservoirHeadsAndConsumerDemands}}

\begin{theoremNoNb}[Theorem \ref{theorem_ExistenceAndUniquenessOfHydraulicStates_ReservoirHeadsAndConsumerDemands}]
    Let $\G = (V,E)$ be a \gls{WDS} and
    $\hm_\Vr$ and $\dm = \dm_\Vc$ 
    known 
    reservoir heads and (consumer) demands, respectively.
    \\There exists exactly one 
    tuple $(\hm_\Vc,\qm)$ 
    of consumer heads and flows 
    such that the triple 
    $(\hm,\qm,\dm) = ([\hm_\Vr^T,\hm_\Vc^T]^T,\qm,\dm)$ 
    of hydraulic states is physically correct \gls{wrt} these side constraints.
\end{theoremNoNb}

\begin{proof}
    Given 
    reservoir heads $\hm_\Vr \in \R^\n{r}$ and 
    (consumer) demands $\dm_\Vc \in \R^\n{c}$, 
    Equation \eqref{align_HydraulicPrinciples_ConservationOfEnergy_MatrixForm_separatedHeads} and
    \eqref{align_HydraulicPrinciples_ConservationOfMass_MatrixForm_separatedHeads} 
    can be transformed into the \gls{SNLE}
    \begin{align}
    \label{align_SNLE_ConsumerHeadsAndFlowsUnkown}
        \underbrace{
        \begin{array}{c@{\hspace{-2pt}}@{\hspace{-2pt}}c}
            \raisebox{-2pt}{
            \text{\tiny $\overbrace{\phantom{................}}^{\n{e} \times \n{c}}$}
            }
            & 
            \raisebox{-2pt}{
            \text{\tiny $\overbrace{\phantom{............}}^{\n{e} \times \n{e}}$}
            }
            \\
            \multicolumn{2}{c}{
            \begin{bmatrix}
                \Bm_{\Vc \cdot}^T    & -\Dm\\
                \Zerom{\n{c}}{\n{c}} & \Bm_{\Vc \cdot}
            \end{bmatrix}
            } 
            \\
            \raisebox{7pt}{
            \text{\tiny $\underbrace{\phantom{................}}_{\n{c} \times \n{c}}$}
            }
            & 
            \raisebox{7pt}{
            \text{\tiny $\underbrace{\phantom{............}}_{\n{c} \times \n{e}}$}
            }
        \end{array}
        }_{
        \in \R^{(\n{e}+\n{c}) \times (\n{c}+\n{e})}
        }
        \cdot
        \underbrace{
        \begin{array}{c}
            \raisebox{-2pt}{
            \text{\tiny $\overbrace{\phantom{..........}}^{\n{c} \times 1}$}
            }
            \\
            \begin{bmatrix}
                \hm_{\Vc} \\
                \qm
            \end{bmatrix}
            \\
            \raisebox{7pt}{
            \text{\tiny $\underbrace{\phantom{..........}}_{\n{e} \times 1}$}
            }
        \end{array}
        }_{
        \in \R^{\n{c}+\n{e}}
        }
        =
        \underbrace{
        \begin{array}{c@{\hspace{0pt}}@{\hspace{0pt}}c}
            \raisebox{-2pt}{
            \text{\tiny \color{white} $\overbrace{\phantom{.....}}^{\n{e} \times 1}$}
            }
            & 
            \\
            \multirow{2}{*}{$
            \begin{bmatrix}
                - \Bm_{\Vr \cdot}^T \hm_\Vr \\
                - 2 \dm
            \end{bmatrix}
            $}
            & 
            \text{$\}$ \tiny $\n{e} \times 1$}
            \\
            & 
            \text{$\}$ \tiny $\n{c} \times 1$}
            \\
            \raisebox{7pt}{
            \text{\tiny \color{white} $\underbrace{\phantom{.....}}_{\n{c} \times 1}$}
            }
            &
        \end{array}
        }_{
        \in \R^{\n{e}+\n{c}}
        }
        .
    \end{align}

    \noindent 
    \textbf{Existence:}
    \textit{Step 1:}
    \textit{There exists a $\qm_\dm \in \R^\n{e}$ such that $\Bm_{\Vc \cdot} \qm_\dm = - 2\dm$ holds.}

    Since by Theorem \ref{theorem_RankOfSubmatricesOfTheIncidenceMatrix}, $\rk(\Bm_{\Vc \cdot}) = \n{c}$ holds,
    by basic results from linear algebra, the matrix $\Bm_{\Vc \cdot} \in \R^{\n{c} \times \n{e}}$ induces a surjective linear map $\R^\n{e} \rightarrow \R^{\n{c}}, ~ \qm \mapsto \Bm_{\Vc \cdot} \qm$. 
    Consequently, $- 2\dm \in \R^\n{c} = \im(\Bm_{\Vc \cdot})$ holds and thus, there exists a $\qm_\dm \in \R^\n{e}$ such that $\Bm_{\Vc \cdot} \qm_\dm = - 2\dm$ holds.

    \begin{remark}[Uniqueness comes only from the coupled systems of equations]
    \label{remark_OnTheConservationOfMass}
        By $\n{c} = \rk(\Bm_{\Vc \cdot}) \leq \min\{\n{c},\n{e}\} \leq \n{c}$, $\min\{\n{c},\n{e}\} = \n{c}$ and thus, $\n{e} \geq \n{c}$ follows. 
        If $\n{e} = \n{c}$ holds, the linear map is not only surjective, but bijective.
        However, in the usual case where $\n{e} > \n{c}$ holds, the linear map can not be injective. 
        Intuitively, the kernel $\ker(\Bm_{\Vc \cdot})$ of $\Bm_{\Vc \cdot}$ consists of all flows $\qm$ which flow in cycles through the \gls{WDS} without any sources and sinks. However, physically, such cycles would contradict the fact that water flows along decreasing heads (cf. Section \ref{subsection_WaterHydraulics_HydraulicPrinciples}).
        Therefore, the uniqueness of a $\qm$ that solves the \gls{SNLE} \eqref{align_SNLE_ConsumerHeadsAndFlowsUnkown} comes only in combination with the second set of equations.
    \end{remark}

    \noindent 
    \textit{Step 2:}
    \textit{Given $\qm_\dm$ from Step 1,
    there exist $\qm_\zerom{} \in \ker(\Bm_{\Vc \cdot})$ and $\hm_\Vc \in \R^\n{c}$, such that}
    \begin{align}
    \label{align_HydraulicPrinciples_ConservationOfEnergy_MatrixForm_separatedHeads_modified}
        \Bm_{\Vc \cdot}^T \hm_\Vc - \Dm (\qm_\dm + \qm_\zerom{}) 
        = 
        - \Bm_{\Vr \cdot}^T \hm_\Vr
    \end{align}

    \noindent \textit{holds.}
    
    In order to prove the claim, we make use of the orthogonal projection
    \begin{align*}
        \Pi: \R^\n{e} \longrightarrow \ker(\Bm_{\Vc \cdot}),~
        \qm \mapsto \Pi(\qm)
    \end{align*}

    \noindent from the (euclidean) vector space $\R^\n{e}$ onto the sub-vector space $\ker(\Bm_{\Vc \cdot}) \subset \R^\n{e}$. The following results are well-known results from linear algebra:
    
    \begin{enumerate}
        \item $\Pi$ is linear,

        \item $\Pi$ is self-adjoint, that is $\langle \Pi(\qm_1),\qm_2 \rangle = \langle \qm_1,\Pi(\qm_2) \rangle$ holds for all $\qm_1,\qm_2 \in \R^\n{e}$,

        \item $\Pi(\qm) = \qm$ \gls{iff} $\qm \in \ker(\Bm_{\Vc \cdot})$ and

        \item $\im(\Pi) = \ker(\Bm_{\Vc \cdot})$ and $\ker(\Pi) = \ker(\Bm_{\Vc \cdot})^\perp = \im(\Bm_{\Vc \cdot}^T)$. 
    \end{enumerate}

    \noindent Based on the orthogonal projection $\Pi$ and the \gls{SNLE} \eqref{align_HydraulicPrinciples_ConservationOfEnergy_MatrixForm_separatedHeads_modified}, we define the function
    \begin{align*}
        \Psi: \ker(\Bm_{\Vc \cdot}) \longrightarrow \ker(\Bm_{\Vc \cdot}),~
        \qm_\zerom{} \mapsto \Pi(f(\qm_\dm + \qm_\zerom{}) - \Bm_{\Vr \cdot}^T \hm_\Vr).
    \end{align*}

    \noindent 
    \textit{Step 2.1:}
    \textit{It sufficies to show that $\Psi$ is surjective.}
    
    Since clearly, $\zerom{\n{e}} \in \ker(\Bm_{\Vc \cdot})$ holds, 
    if $\Psi$ is surjective, 
    there exists a $\qm_\zerom{} \in \ker(\Bm_{\Vc \cdot})$ 
    such that $\Psi(\qm_\zerom{}) = \Pi(f(\qm_\dm + \qm_\zerom{}) - \Bm_{\Vr \cdot}^T \hm_\Vr) = \zerom{\n{e}}$ holds. 
    Consequently, 
    $f(\qm_\dm + \qm_\zerom{}) - \Bm_{\Vr \cdot}^T \hm_\Vr \in \ker(\Pi) = \ker(\Bm_{\Vc \cdot})^\perp = \im(\Bm_{\Vc \cdot}^T)$ holds. 
    That is, there exists a $\hm_\Vc \in \R^\n{c}$, such that 
    \begin{align*}
        \Bm_{\Vc \cdot}^T \hm_\Vc 
        = 
        f(\qm_\dm + \qm_\zerom{}) - \Bm_{\Vr \cdot}^T \hm_\Vr
        =
        \Dm (\qm_\dm + \qm_\zerom{}) - \Bm_{\Vr \cdot}^T \hm_\Vr
    \end{align*}

    \noindent holds.
    Therefore, if $\Psi$ is surjective, 
    there exist $\qm_\zerom{} \in \ker(\Bm_{\Vc \cdot})$ and $\hm_\Vc \in \R^\n{c}$, such that \eqref{align_HydraulicPrinciples_ConservationOfEnergy_MatrixForm_separatedHeads_modified} holds.
    \\
    \\\textit{Step 2.2.:}
    \textit{$\Psi$ is surjective.}
    
    We show the surjectivity of $\Psi$ using the following theorem:

    \begin{theorem}[General existence theorem in finite dimensional spaces (Theorem 16.4 in \cite{Schweizer2013OrdinaryAndPartialDifferentialEquations_german})]
    \label{theorem_GeneralExistenceTheoremInFiniteDimensionalSpaces}
        Let $\Psi: \R^n \rightarrow \R^n$ be continuous and coercive, that is,
        \begin{align*}
            \lim_{||\xm||_2 \rightarrow \infty}
            \frac{
            \langle \Psi(\xm) , \xm \rangle
            }{
            ||\xm||_2
            }
            =
            \infty.
        \end{align*}

        \noindent Then $\Psi$ is surjective, i.e., for each $\bm \in \R^n$ there exists a $\xm \in \R^n$ such that $\Psi(\xm) = \bm$ holds.
    \end{theorem}

    \noindent As a subspace of $\R^\n{e}$, $\ker(\Bm_{\Vc \cdot}) \subset \R^\n{e}$ is itself an $n$-dimensional (euclidean) vector space with -- using the rank-nullity theorem and the previous observation $\rk(\B_{\Vc \cdot}) = \n{c}$ -- dimension $n = \dim(\ker(\Bm_{\Vc \cdot})) = \n{e} - \dim(\im(\Bm_{\Vc \cdot})) = \n{e} - \rk(\Bm_{\Vc \cdot}) = \n{e} - \n{n}$.
    Therefore, we can apply Theorem \ref{theorem_GeneralExistenceTheoremInFiniteDimensionalSpaces} to $\Psi: \ker(\Bm_{\Vc \cdot}) \longrightarrow \ker(\Bm_{\Vc \cdot})$ as defined above.

    \textit{$\Psi$ is continuous:}
    Since $\Pi$ is linear, it is continuous. Moreover, $f$ is continuous. Therefore, so is $\Psi$ as a composition of continuous functions.

    \textit{$\Psi$ is coercive:}
    By
    (1) linearity of the scalar product,
    for any $\qm_\zerom{} \in \ker(\Bm_{\Vc \cdot})$ with $||\qm_\zerom{}||_2 \neq 0$,
    we obtain

    \begin{align*}
        \frac{
        \langle \Psi(\qm_\zerom{}) , \qm_\zerom{} \rangle
        }{
        ||\qm_\zerom{}||_2
        }
        \overset{\text{(1)}}{=}&~
        \frac{
        \langle \Psi(\qm_\zerom{}) - \Psi(\zerom{\n{e}}), \qm_\zerom{} \rangle
        }{
        ||\qm_\zerom{}||_2
        }
        +
        \frac{
        \langle \Psi(\zerom{\n{e}}), \qm_\zerom{} \rangle
        }{
        ||\qm_\zerom{}||_2
        }.
    \end{align*}

    \noindent 
    By
    (1) definition of $\Psi$,
    (2) linearity of $\Pi$,
    (3) self-adjointness of $\Pi$,
    (4) the fact that $\Pi_{| \ker(\Bm_{\Vc \cdot})} = \id$ holds,
    (5) basic transformations and
    (6) Lemma \ref{lemma_TheNonLinearityDefinesAStrictilyMonotoneOperator} with $\qm_1 = \qm_\dm + \qm_\zerom{} \in \R^\n{e}$ and $\qm_2 = \qm_\dm \in \R^\n{e}$, 
    we obtain
    
    \begin{align*}
        \frac{
        \langle \Psi(\qm_\zerom{}) - \Psi(\zerom{\n{e}}), \qm_\zerom{} \rangle
        }{
        ||\qm_\zerom{}||_2
        }
        \overset{\text{(1)}}{=}&~
        \frac{
        \langle ~
        \Pi(f(\qm_\dm + \qm_\zerom{}) - \Bm_{\Vr \cdot}^T \hm_\Vr)
        -
        \Pi(f(\qm_\dm + \zerom{\n{e}}) - \Bm_{\Vr \cdot}^T \hm_\Vr)
        ~,~ 
        \qm_\zerom{} 
        ~ \rangle
        }{
        ||\qm_\zerom{}||_2
        }
        \\
        \overset{\text{(2)}}{=}&~
        \frac{
        \langle ~
        \Pi(f(\qm_\dm + \qm_\zerom{}) - f(\qm_\dm)) 
        ~,~ 
        \qm_\zerom{}
        ~ \rangle
        }{
        ||\qm_\zerom{}||_2
        }
        \\
        \overset{\text{(3)}}{=}&~
        \frac{
        \langle ~
        f(\qm_\dm + \qm_\zerom{}) - f(\qm_\dm)
        ~,~ 
        \Pi(\qm_\zerom{})
        ~ \rangle
        }{
        ||\qm_\zerom{}||_2
        }
        \\
        \overset{\text{(4)}}{=}&~
        \frac{
        \langle ~
        f(\qm_\dm + \qm_\zerom{}) - f(\qm_\dm)
        ~,~ 
        \qm_\zerom{}
        ~ \rangle
        }{
        ||\qm_\zerom{}||_2
        }
        \\
        \overset{\text{(5)}}{=}&~
        \frac{
        \langle ~
        f(\qm_\dm + \qm_\zerom{}) - f(\qm_\dm)
        ~,~ 
        (\qm_\dm + \qm_\zerom{}) - \qm_\dm
        ~ \rangle
        }{
        ||\qm_\zerom{}||_2
        }
        \\
        \overset{\text{(6)}}{=}&~
        \frac{1}{||\qm_\zerom{}||_2}
        \sum_{e \in E}
        x ~ r_e ~ q_{\zerom{} e}^2
        \int_{0}^{1}
        |q_{\dm e} + t ~q_{\zerom{} e}|^{x-1}
        ~dt.
    \end{align*}

    \noindent Let us choose $e = \argmax_{e \in E} |q_{\zerom{} e}|$. 
    Then if $||\qm_\zerom{}||_2 = \left( \sum_{e \in E} |q_{\zerom e}|^2 \right)^{\frac{1}{2}} \rightarrow \infty$ holds, so does $|q_{\zerom e}| \rightarrow \infty$.
    Therefore, \gls{wlog}, we can assume that $|\frac{1}{4} q_{\zerom{} e}| > |q_{\dm e}|$ holds.
    
    Consequently, by
    (1) basic transformations,
    (2) making usage of the maximum $|q_{\zerom e}| = \max_{e \in E} |q_{\zerom e}|$,
    (3) reversed triangle inequaility,
    (4) the fact that 
    $t |q_{\zerom{} e}| - |q_{\dm e}| > t |q_{\zerom{} e}| - \frac{1}{4} |q_{\zerom{} e}| > 0$  
    holds for $|q_{\zerom{} e}|$ large enough and all $t \in [\frac{1}{2},1]$ and
    (5) the fact that $t^x > t$ holds for all $t,x \in (0,1)$ and
    we obtain
    \begin{align*}
        \frac{
        \langle \Psi(\qm_\zerom{}) - \Psi(\zerom{\n{e}}), \qm_\zerom{} \rangle
        }{
        ||\qm_\zerom{}||_2
        }
        =&~
        \frac{1}{||\qm_\zerom{}||_2}
        \sum_{e \in E}
        x ~ r_e ~ q_{\zerom{} e}^2
        \int_{0}^{1}
        |q_{\dm e} + t ~q_{\zerom{} e}|^{x-1}
        ~dt
        \\
        \overset{\text{(1)}}{\geq}&~
        \frac{1}{||\qm_\zerom{}||_2}
        \sum_{e \in E}
        x ~ r_e ~ q_{\zerom{} e}^2
        \int_{\frac{1}{2}}^{1}
        |t ~q_{\zerom{} e} + q_{\dm e}|^{x-1}
        ~dt
        \\
        \overset{\text{(2)}}{\geq}&~
        \frac{
        x ~ r_e ~ q_{\zerom{} e}^2
        }{
        \n{e}^{\frac{1}{2}} ~ |q_{\zerom{} e}|
        }
        \int_{\frac{1}{2}}^{1}
        |t ~q_{\zerom{} e} + q_{\dm e}|^{x-1}
        ~dt
        \\
        \overset{\text{(3)}}{\geq}&~
        \frac{
        x ~ r_e ~ q_{\zerom{} e}^2
        }{
        \n{e}^{\frac{1}{2}} ~ |q_{\zerom{} e}|
        }
        \int_{\frac{1}{2}}^{1}
        \big| ~ 
        |t ~q_{\zerom{} e}| 
        - 
        |q_{\dm e}| 
        ~ \big|^{x-1}
        ~dt
        \\
        \overset{\text{(4)}}{>}&~
        \frac{
        x ~ r_e ~ q_{\zerom{} e}^2
        }{
        \n{e}^{\frac{1}{2}} ~ |q_{\zerom{} e}|
        }
        \int_{\frac{1}{2}}^{1}
        \left| ~ 
        t |q_{\zerom{} e}| 
        - 
        \frac{1}{4} |q_{\zerom{} e}|
        ~ \right|^{x-1}
        ~dt
        \\
        \overset{\text{(1)}}{=}&~
        \frac{
        x ~ r_e ~ q_{\zerom{} e}^2 ~ |q_{\zerom{} e}|^{x-1} 
        }{
        \n{e}^{\frac{1}{2}} ~ |q_{\zerom{} e}|
        }
        \int_{\frac{1}{2}}^{1}
        \left( t - \frac{1}{4} \right)^{x-1} 
        ~dt
        \\
        \overset{\text{(5)}}{>}&~
        \frac{
        x ~ r_e ~ q_{\zerom{} e}^2 ~ |q_{\zerom{} e}|^{x-1} 
        }{
        \n{e}^{\frac{1}{2}} ~ |q_{\zerom{} e}|
        }
        \int_{\frac{1}{2}}^{1}
        \left( t - \frac{1}{4} \right)
        ~dt
        \\
        \overset{\text{(1)}}{=}&~
        \frac{
        x ~ r_e ~ |q_{\zerom{} e}|^x
        }{
        4 \n{e}^{\frac{1}{2}}
        }.
    \end{align*}

    \noindent Since the right-hand side converges towards infinity as $||\qm_\zerom{}||_2 \rightarrow \infty$, we can conclude that
    \begin{align*}
        \lim_{||\qm_\zerom{}||_2 \rightarrow \infty}
        \frac{
        \langle \Psi(\qm_\zerom{}) - \Psi(\zerom{\n{e}}), \qm_\zerom{} \rangle
        }{
        ||\qm_\zerom{}||_2
        }
        =
        \infty
    \end{align*}

    \noindent holds. Therefore, it suffices to show that the remaining term does not converge towards minus infinity. Indeed, by
    Cauchy-Schwarz inequality,
    \begin{align*}
        \Big |
        \frac{
        \langle \Psi(\zerom{\n{e}}), \qm_\zerom{} \rangle
        }{
        ||\qm_\zerom{}||_2
        }
        \Big |
        = 
        \frac{
        | \langle \Psi(\zerom{\n{e}}), \qm_\zerom{} \rangle |
        }{
        ||\qm_\zerom{}||_2
        }
        \leq 
        \frac{
        || \Psi(\zerom{\n{e}}) ||_2 \cdot || \qm_\zerom{} ||_2
        }{
        ||\qm_\zerom{}||_2
        }
        = 
        || \Psi(\zerom{\n{e}}) ||_2
        < 
        \infty
    \end{align*}

    \noindent holds.
    In summary, $\Psi$ is continuous and coercive, and therefore, by Theorem \ref{theorem_GeneralExistenceTheoremInFiniteDimensionalSpaces}, it is surjective.
    \\
    \\\textit{Step 3:}
    \textit{
    Given $\qm_\dm$ from Step 1 and $\qm_\zerom{} \in \ker(\Bm_{\Vc \cdot}) \subset \R^\n{e}$ and $\hm_{\Vc \cdot} \in \R^\n{c}$ from Step 2,
    the vector $(\hm_\Vc, \qm) := (\hm_\Vc, \qm_\dm + \qm_\zerom{})$ is solution to the \gls{SNLE} \eqref{align_SNLE_ConsumerHeadsAndFlowsUnkown}.}

    The vector $(\hm_\Vc, \qm) := (\hm_\Vc, \qm_\dm + \qm_\zerom{})$ satisfies
    \begin{align*}
        \Bm_{\Vc \cdot}^T \hm_\Vc - \Dm \qm
        =
        \Bm_{\Vc \cdot}^T \hm_\Vc - \Dm (\qm_\dm + \qm_\zerom{}) 
        =
        - \Bm_{\Vr \cdot}^T \hm_\Vr
    \end{align*}

    \noindent by \textit{Step 2} and 
    \begin{align*}
        \Bm_{\Vc \cdot} \qm 
        =
        \Bm_{\Vc \cdot} (\qm_\dm + \qm_\zerom{}) 
        =
        \Bm_{\Vc \cdot} \qm_\dm +  \Bm_{\Vc \cdot} \qm_\zerom{}
        =
        - 2\dm + \zerom{\n{c}}
        = 
        - 2\dm
    \end{align*}

    \noindent by \textit{Step 1} and the fact that $\qm_\zerom{} \in \ker(\Bm_{\Vc \cdot})$ holds, and therefore is a solution to the \gls{SNLE} \eqref{align_SNLE_ConsumerHeadsAndFlowsUnkown}.
    \\
    \\In conclusion, 
    the tuple
    $(\hm_\Vc,\qm) := (\hm_\Vc,\qm_\dm + \qm_\zerom{})$ 
    solves the 
    \gls{SNLE} \eqref{align_SNLE_ConsumerHeadsAndFlowsUnkown}
    and therefore, 
    the triple 
    $(\hm,\qm,\dm) = ([\hm_\Vr^T,\hm_\Vc^T]^T,\qm_\dm + \qm_\zerom{},\dm)$ 
    is physically correct \gls{wrt} the given side constraints.
    \\
    \\
    Before investigating on the uniqueness, let us summarize some general results regarding the solution $\qm_\dm$ of the \gls{SLE} $\Bm_{\Vc \cdot} \qm = - 2\dm$ from \textit{Step 1}:

    By basic results from linear algebra, we know that the solution space $L(\Bm_{\Vc \cdot},- 2\dm)$ of the \gls{SLE} $\Bm_{\Vc \cdot} \qm = - 2\dm$ can be written as
    \begin{align*}
        L(\Bm_{\Vc \cdot},- 2\dm)
        =
        \qm_\dm + \ker(\Bm_{\Vc \cdot})
        \subset 
        \R^\n{e}.
    \end{align*}
    
    \noindent Even more, it is a well-known result that the linear map $\R^\n{e} \rightarrow \R^{\n{c}}, ~ \qm \mapsto \Bm_{\Vc \cdot} \qm$ induces a decomposition of the vector space $\R^\n{e}$ as the direct sum
    \begin{align*}
        \R^\n{e} = \ker(\Bm_{\Vc \cdot}) \oplus \im(\Bm_{\Vc \cdot}^T)
    \end{align*}

    \noindent of the kernel $\ker(\Bm_{\Vc \cdot})$ of the matrix $\Bm_{\Vc \cdot}$ and the image $\im(\Bm_{\Vc \cdot}^T)$ of the transposed matrix $\Bm_{\Vc \cdot}^T$.
    As a consequence of the above considerations, each $\qm \in \R^\n{e}$ can be written as 
    \begin{align*}
        \qm = \qm_\dm + \qm_\zerom{}
        \text{ with }
        \qm_\dm \in \im(\Bm_{\Vc \cdot}^T) \perp \ker(\Bm_{\Vc \cdot})
        \text{ and }
        \qm_\zerom{} \in \ker(\Bm_{\Vc \cdot}).
    \end{align*}

    \noindent This observation allows us to investigate in the uniqueness of the solution found above.
    \\
    \\\textbf{Uniqueness:}
    Let us assume 
    that there exist two
    tuples $(\hm_{1 \Vc},\qm_{1}) \neq (\hm_{2 \Vc},\qm_{2})$ 
    of consumer heads and flows 
    such that the triples 
    $([\hm_\Vr^T,\hm_{1 \Vc}^T]^T,\qm_1,\dm)$ and $([\hm_\Vr^T,\hm_{2 \Vc}^T]^T,\qm_2,\dm)$ 
    of hydraulic states are physically correct \gls{wrt} these side constraints.
    
    As elaborated above, for $\qm_1$ and $\qm_2$ there exist $\qm_{\zerom{} 1} \in \ker(\Bm_{\Vc \cdot})$ and $\qm_{\zerom{} 2} \in \ker(\Bm_{\Vc \cdot})$ such that 
    \begin{align*}
        \qm_1 = \qm_\dm + \qm_{\zerom{} 1} \text{ and }
        \qm_2 = \qm_\dm + \qm_{\zerom{} 2}
    \end{align*}

    \noindent holds. Moreover, since  $([\hm_\Vr^T,\hm_{1 \Vc}^T]^T,\qm_1,\dm)$ and $([\hm_\Vr^T,\hm_{2 \Vc}^T]^T,\qm_2,\dm)$ are physically correct,
    \begin{align*}
        \Bm_{\Vc \cdot} \qm_1 
        = 
        - 2\dm
        =
        \Bm_{\Vc \cdot} \qm_2 \text{ and }
        \Bm_{\Vc \cdot}^T \hm_{1 \Vc} - \Dm(\qm_1) \qm_1 
        =
        - \Bm_{\Vr \cdot}^T \hm_\Vr
        =
        \Bm_{\Vc \cdot}^T \hm_{2 \Vc} - \Dm(\qm_2) \qm_2 
    \end{align*}

    \noindent and therefore, especially, by definition of the non-linearity $f$ (Lemma \ref{lemma_TheNonLinearityDefinesAStrictilyMonotoneOperator}),
    \begin{align*}
        f(\qm_1)  - f(\qm_2) 
        =
        \Dm(\qm_1) \qm_1  - \Dm(\qm_2) \qm_2 
        =
        \Bm_{\Vc \cdot}^T \hm_{1 \Vc} -  \Bm_{\Vc \cdot}^T \hm_{2 \Vc}
        =
        \Bm_{\Vc \cdot}^T (\hm_{1 \Vc} - \hm_{2 \Vc})
    \end{align*}

    \noindent holds.
    \\
    \\\textit{Step 1:} If $\qm_1 \neq \qm_2$ holds,
    by
    (1) the considerations above,
    (2) basic transformations and
    (3) the fact that $\qm_{\zerom{} 1}, \qm_{\zerom{} 2} \in \ker(\Bm_{\Vc \cdot})$ holds,
    we obtain
    \begin{align*}
        \langle 
        f(\qm_1) - f(\qm_2), \qm_1 - \qm_2
        \rangle
        \overset{\text{(1)}}{=}&~
        \langle 
        \Bm_{\Vc \cdot}^T (\hm_{1 \Vc} - \hm_{2 \Vc}) , \qm_{\zerom{} 1} - \qm_{\zerom{} 2}
        \rangle
        \\
        \overset{\text{(2)}}{=}&~
        \langle 
        \hm_{1 \Vc} - \hm_{2 \Vc} , \Bm_{\Vc \cdot} (\qm_{\zerom{} 1} - \qm_{\zerom{} 2})
        \rangle
        \\
        \overset{\text{(3)}}{=}&~
        \langle 
        \hm_{1 \Vc} - \hm_{2 \Vc} , \zerom{\n{c}}
        \rangle
        \\
        \overset{\text{(2)}}{=}&~
        0,
    \end{align*}

    \noindent a contradiction to the strict monotonicity of the non-linearity $f$ (cf. Lemma \ref{lemma_TheNonLinearityDefinesAStrictilyMonotoneOperator}).
    Therefore, $\qm_1 = \qm_2$ must hold.
    \\
    \\\textit{Step 2:} If $\hm_{1 \Vc} \neq \hm_{2 \Vc}$ holds,
    by 
    (1) the considerations above and
    (2) the fact that $\qm_1 = \qm_2$ holds by \textit{Step 1},
    we obtain
    \begin{align*}
        \Bm_{\Vc \cdot}^T \hm_{1 \Vc}
        \overset{\text{(1)}}{=}~
        f(\qm_1) - \Bm_{\Vr \cdot}^T \hm_\Vr
        \overset{\text{(1)}}{=}~
        f(\qm_2) - \Bm_{\Vr \cdot}^T \hm_\Vr
        \overset{\text{(1)}}{=}~
        \Bm_{\Vc \cdot}^T \hm_{2 \Vc}.
    \end{align*}

    \noindent Consequently, $\hm_{1 \Vc}, \hm_{2 \Vc} \in \im(\Bm_{\Vc \cdot}^T)$ have the same image under $\Bm_{\Vc \cdot}^T$. 
    However, 
    since by Theorem \ref{theorem_RankOfSubmatricesOfTheIncidenceMatrix}, $\rk(\Bm_{\Vc \cdot}) = \n{c}$ holds,
    by basic results from linear algebra, the matrix $\Bm_{\Vc \cdot}^T \in \R^{\n{e} \times \n{c}}$ induces an injective linear map $\R^\n{c} \longrightarrow \R^\n{e}, \hm_\Vc \longmapsto \Bm_{\Vc \cdot}^T \hm_\Vc$. 
    Therefore, $\hm_{1 \Vc} = \hm_{2 \Vc}$ must hold.
\end{proof}

\end{document}